\documentclass[11pt,reqno]{amsart}
\usepackage{latexsym,amssymb,amsmath,multicol,rotating,
lscape}
 2

\newtheorem{thm}{Theorem}[section]
\newtheorem{lem}[thm]{Lemma}
\newtheorem{prop}[thm]{Proposition}

\newcommand{\thmref}[1]{Theorem~\ref{#1}}

\newcommand{\propref}[1]{Proposition~\ref{#1}}

\theoremstyle{remark}
\newtheorem{rmk}{Remark}[section]

\begin{document}

\title[quaternary and octonary quadratic forms]
{Representations of an integer by some quaternary and octonary quadratic forms} 

\author{B. Ramakrishnan, Brundaban Sahu and Anup Kumar Singh}
\address[B. Ramakrishnan and Anup Kumar Singh]{Harish-Chandra Research Institute, HBNI, 
       Chhatnag Road, Jhunsi,
     Allahabad -     211 019,
   India.}
\address[Brundaban Sahu]
{School of Mathematical Sciences, National Institute of Science  
Education and Research, Bhubaneswar, HBNI, 
Via - Jatni, Khurda, Odisha - 752 050,
India.}

\email[B. Ramakrishnan]{ramki@hri.res.in}
\email[Brundaban Sahu]{brundaban.sahu@niser.ac.in}
\email[Anup Kumar Singh]{anupsingh@hri.res.in}

\dedicatory{Dedicated to Professor V. Kumar Murty on the occasion of his 60th birthday}
\date{\today}

\begin{abstract}
In this paper we consider certain quaternary quadratic forms and octonary quadratic forms and by using the theory of modular forms, 
we find formulae for the number of representations of a positive integer by these quadratic forms. 
%In this paper we use the theory of modular forms to find formulas for the number of representations of a positive integer by the  quadratic forms  
% $x_1^2+x_1x_2+jx_2^2+l(x_3^2+x_3x_4+jx_4^2),$ where $(j,l)\in \lbrace (5,1),(5,2),(4,2),(3,2),(3,3),(2,2),(2,3),(2,4),(1,2),(1,3),(1,4),(1,5)\rbrace$.
\end{abstract}

\maketitle

\section{Introduction}
In this paper we consider two types of quadratic forms, viz., quaternary and octonary forms. In the first part, we deal with quaternary quadratic forms 
of the following type given by 
\linebreak 
${\mathcal Q}_{a,\ell} = {\mathcal Q}_a \oplus \ell {\mathcal Q}_a : x_1^2 + x_1x_2 + a x_2^2 + \ell (x_3^2 + x_3x_4 + a x_4^2)$,
where ${\mathcal Q}_a$ is the quadratic form $x_1^2 + x_1x_2 + a x_2^2$.  Let $R_{a,\ell}(n)$ denote the number of ways of representing a positive integer $n$ by the quadratic form ${\mathcal Q}_{a,\ell}$. i.e., 
$$
R_{a,\ell}(n) := {\rm card}\left\{(x_1, x_2, x_3, x_4) \in {\mathbb Z}^{4} :n = x_1^2+x_1x_2+a x_2^2+ \ell (x_3^2+x_3x_4+ a x_4^2)\right\}.
$$
One of the main results of this paper is to find formulas for $R_{a,\ell}(n)$,  $(a,\ell) \in \textbf{A}$, where $\textbf{A} = \{(1,2), (1,3), (1,4), (1,5), (2,2), (2,3), (2,4), (3,2), (3,3), (4,2), (5,1), (5,2)\}$.  Let us mention a brief account of similar results obtained so far. 
S. Ramanujan was the first to observe the following identity:
%gave an identity which is equivalent to the formula for $R_{1,1}(n)$, namely,
\begin{equation}\label{rama1}
\left(\sum_{m=-\infty}^{\infty}\sum_{n=-\infty}^{\infty} q^{m^2+mn+n^2}\right)^2 = - \frac{1}{2}E_2(z) + \frac{3}{2}E_2(3z).
\end{equation}
(See \cite[pp. 402--403]{Andrews-Berndt}, \cite[p.460, Entry 3.1]{Berndt} for details.) Since 
$$
\left(\sum_{m=-\infty}^{\infty}\sum_{n=-\infty}^{\infty} q^{m^2+mn+n^2}\right)^2 = 1 + \sum_{n=1}^\infty R_{1,1}(n) q^n,
$$
by comparing the $n$-th Fourier coefficients in \eqref{rama1}, one gets 
\begin{equation}\label{rama2}
R_{1,1}(n) = 12 \sigma(n) - 36 \sigma(n/3).
\end{equation}
In the above, $E_2(z)$ denotes the Eisenstein series of weight $2$ on $SL_2({\mathbb Z})$ which is  given by 
\begin{equation}\label{e2}
E_2(z) = 1 - 24\sum_{n\ge 1} \sigma(n) q^n.
\end{equation}
Note that $E_2(z)$ is a quasimodular form. 
%This fundamental quasimodular form will be used in our results.
Here $q=e^{2 \pi iz}$, $z\in \mathbb{H}$, where $\mathbb{H}=\left\{z\in \mathbb{C}: Im(z) > 0\right.\}$. 
We remark that Ramanujan gave the identity \eqref{rama1} without proof. For a proof of \eqref{rama2} one can refer to Huard et al. 
[\cite{huard}, Theorem 13] or Lomadze[10]. Formulas for $R_{a,1}(n)$ for $a = 2,3,4,6,7$ are known due to the works of 
several authors using different methods. In this paper, we consider the case $a =5$. Further, we also consider the case $\ell >1$ for a few values 
of $\ell$. More precisely, for the pairs $(a,\ell)$ belonging to the set $\textbf{A}$. In the following table, we list the present work and also the earlier works done in this direction. 

\bigskip

\begin{center}
\begin{tabular}{|c|c|c|c|}
\hline
%&&&\\
Present work & Earlier works &Author(s) & References\\
$(a,\ell)$ &$(a,1)$ &(earlier works)&\\ 
%&&\\
\hline \hline 
%&&&\\
(1,2), (1,3), (1,4), (1,5)&(1,1) &  Huard et al., Lomadze  &  \cite{{huard}, {lomadze}} \\
\hline
(2,2), (2,3), (2,4)&(2,1) &  Ramanujan, Chan-Ong, Williams &  \cite{{Andrews-Berndt}, {Chan-ong}, {Williams}} \\
\hline
(3,2), (3,3)&(3,1) &  Chan-Cooper &  \cite{Chan-Cooper} \\
\hline
(4,2)&(4,1) &  Cooper-Ye  &  \cite{Cooper-Ye}  \\
\hline
(5,1), (5,2) & -- &-- &--\\
\hline
-- &   (6,1) &  Chan-Cooper & \cite{Chan-Cooper} \\
\hline
--&(7,1) & Dongxi Ye  &  \cite{Ye}  \\
%&&&\\
\hline 
\end{tabular}

\smallskip
Table 1.
\end{center}

\smallskip

Some of the formulas for $R_{a,\ell}(n)$ involve only the divisor function $\sigma(n)$, namely the cases $(a,\ell) = (1,2), (1,4), (3,2)$. In these cases 
it is possible to get formulas for the number of representations of the quadratic forms in eight variables defined by ${\mathcal Q}_{a,\ell;j} :=$ 
${\mathcal Q}_{a,\ell} \oplus j {\mathcal Q}_{a,\ell}$ using the convolution sums method. We note that this method (doubling the quadratic form with 
coefficients) can be considered in general, here for simplicity we have considered only the following 7 cases: $(a,\ell,j) = 
(1,2,1), (1,2,2), (1,2,3), (1,2,4), (1,4,1), (1,4,2), (3,2,1)$. To be precise, in these cases mentioned above, the formulas do not involve too 
many coefficients coming from the cusp forms. 

\bigskip

In the second part of this article, we consider the following octonary quadratic forms (with coefficients $1,2,4,8$):
\begin{equation}\label{octo}
\sum_{r=1}^i x_r^2 + 2 \sum_{r=i+1}^{i+j} x_r^2 + 4 \sum_{r=i+j+1}^{i+j+k} x_r^2 + 8 \sum_{r=i+j+k+1}^{i+j+k+l} x_r^2, 
\end{equation}
for all partitions $i+j+k+l =8$, $i,j,k,l\ge 0$. There are a total of 165 such quadratic forms, and out of which 81 quadratic forms (corresponding to $ i=0 $ or $ l=0 $) have already been considered  by several authors \cite{{a-a-w}, {a-a-w1}, {octonary-ijnt}}. In the second part, we consider the remaining 84 quadratic forms  and give formulas for the corresponding representation numbers. All these 84 quadratic forms are listed as quadruples $(i,j,k,l)$ 
(corresponding to $i\not= 0$ and  $l\not= 0$) in Table 2 below. 

\smallskip

\begin{center}
{\small 
\begin{tabular}{|c|c|}
\hline
$(i,j,k,l)$ & Type\\
\hline 
(1,0,1,6), (1,0,3,4), (1,0,5,2), (1,1,1,5), (1,1,3,3), (1,1,5,1), (1,2,1,4), (1,2,3,2), (1,3,1,3), (1,3,3,1),&\\
(1,4,1,2), (1,5,1,1), (2,0,0,6), (2,0,2,4), (2,0,4,2), (2,1,0,5),
(2,1,2,3), (2,1,4,1), (2,2,0,4), (2,2,2,2),& I\\
(2,3,0,3), (2,3,2,1), (2,4,0,2), (2,5,0,1), (3,0,1,4), (3,0,3,2), (3,1,1,3), (3,1,3,3), (3,2,1,2), (3,3,1,1),&\\
(4,0,0,4), (4,0,2,2), (4,1,0,3), (4,1,2,1), (4,2,0,2), (4,3,0,1),
(5,0,1,2), (5,1,1,1), (6,0,0,2), (6,1,0,1)&\\
&\\
\hline
&\\
(1,0,0,7), (1,0,2,5), (1,0,4,3), (1,0,6,1), (1,1,0,6), (1,1,2,4), (1,1,4,2), (1,2,0,5), (1,2,2,3), (1,2,4,1),&\\
(1,3,0,4), (1,3,2,2), (1,4,0,3), (1,4,2,1), (1,5,0,2), (1,6,0,1),
(2,0,1,5), (2,0,3,3), (2,0,5,1), (2,1,1,4),& II\\
(2,1,3,2), (2,2,1,3), (2,2,3,1), (2,3,1,2), (2,4,1,1), (3,0,0,5), (3,0,2,3), (3,0,4,1), (3,1,0,4), (3,1,2,2),&\\
(3,2,0,3), (3,2,2,1), (3,3,0,2), (3,4,0,1), (4,0,1,3), (4,0,3,1), (4,1,1,2), (4,2,1,1), (5,0,0,3), (5,0,2,1),&\\
(5,1,0,2), (5,2,0,1), (6,0,1,1), (7,0,0,1)&\\
\hline 
\end{tabular}
}

\smallskip
Table 2.
\end{center}

There are several methods used in the literature to obtain results of this type.  In this paper, we  use the theory of modular forms to prove our 
formulas. We first obtain the level and character of the modular forms corresponding to these quadratic forms. Then by using 
explicit bases for the spaces of modular forms, we deduce our formulas. 

\section{Preliminaries and statement of results}

As we use the theory of modular forms, we shall first present some preliminary facts on modular forms. For $k\in \frac{1}{2}{\mathbb Z}$,
let $M_k(\Gamma_0(N),\chi)$ denote the space of modular forms of weight $k$ for the congruence subgroup $\Gamma_0(N)$ with character $\chi$ and $S_k(\Gamma_0(N), \chi)$ be the subspace of cusp forms of weight $k$ for $\Gamma_0(N)$ with character $\chi$. We assume $4\vert N$ when $k$ is not an integer and in that case, the character $\chi$ which is a Dirichlet character modulo $N$, is an even character. When $\chi$ is the trivial (principal) character modulo $N$, we shall denote the spaces by $M_k(\Gamma_0(N))$ and $S_k(\Gamma_0(N))$ respectively. Further, when $k\ge 4$ is an integer and $N=1$, we shall denote the vector spaces by $M_k$ and $S_k$ respectively. 

For an integer $k \ge 4,$ let $E_k$ denote the normalized Eisenstein series of weight $k$ in $M_k$ given by 
$$
E_k(z) = 1 - \frac{2k}{B_k}\sum_{n\ge 1} \sigma_{k-1}(n) q^n,
$$
where $q=e^{2 i\pi z}$, $\sigma_r(n)$ is the sum of the $r$th powers of the positive divisors of $n$, and $B_k$ is the $k$-th Bernoulli number defined by $\displaystyle{\frac{x}{e^x-1} = \sum_{m=0}^\infty \frac{B_m}{m!} x^m}$.

The classical theta function which is fundamental to the theory of modular forms of half-integral weight is defined by 
\begin{equation}\label{theta}
\Theta(z) = \sum_{n\in {\mathbb Z}} q^{n^2},
\end{equation}
and is a modular form in the space $M_{1/2}(\Gamma_0(4))$. Another function which is mainly used in our work is the Dedekind eta function 
$\eta(z)$, which is  defined by 
\begin{equation}\label{eta}
\eta(z)=q^{1/24} \prod_{n\ge1}(1-q^n).
\end{equation}
An eta-quotient is a finite product of integer powers of $\eta(z)$ and we denote it as follows. 
\begin{equation}\label{eta-q} 
\prod_{i=1}^s \eta^{r_i}(d_i z) := d_1^{r_1} d_2^{r_2} \cdots d_s^{r_s},
\end{equation}
where $d_i$'s are positive integers and $r_i$'s are non-zero integers.

Suppose that $\chi$ and $\psi$ are primitive Dirichlet characters with conductors $M$ and $N$, respectively. For a positive integer $k$, let 
\begin{equation}\label{eisenstein}
E_{k,\chi,\psi}(z) :=  c_0 + \sum_{n\ge 1}\left(\sum_{d\vert n} \psi(d) \cdot \chi(n/d) d^{k-1}\right) q^n,
\end{equation}
where 
$$
c_0 = \begin{cases}
0 &{\rm ~if~} M>1,\\
- \frac{B_{k,\psi}}{2k} & {\rm ~if~} M=1,
\end{cases}
$$
and $B_{k,\psi}$ denotes generalized Bernoulli number with respect to the character $\psi$. 
Then, the Eisenstein series $E_{k,\chi,\psi}(z)$ belongs to the space $M_k(\Gamma_0(MN), \chi/\psi)$, provided $\chi(-1)\psi(-1) = (-1)^k$ 
and $MN\not=1$. When $\chi=\psi =1$ (i.e., when $M=N=1$) and $k\ge 4$, we have $E_{k,\chi,\psi}(z) = E_k(z)$, the normalized Eisenstein series of integer weight $k$ as defined before. We refer to \cite{{miyake}, {Stein}} for details. 
We give a notation to the inner sum in \eqref{eisenstein}:
\begin{equation}\label{divisor}
\sigma_{k-1;\chi,\psi}(n) := \sum_{d\vert n} \psi(d) \cdot \chi(n/d) d^{k-1}.
\end{equation}

Let ${\mathbb N}$ and ${\mathbb N}_0$ 
%and ${\mathbb Z}$ 
denote the set of positive integers and non-negative integers respectively. 
%and integers 
For $a_1, \ldots, a_8 \in {\mathbb N}$ and $n\in {\mathbb N}_0$, we define 
$$
N(a_1, \ldots, a_8;n) := {\rm card}\left\{(x_1,\ldots, x_8)\in {\mathbb Z}^8 \vert n = a_1 x_1^2 + \cdots + a_8 x_8^2\right\}.
$$
Note that $N(a_1, \ldots, a_8;0) =1$. Without loss of generality we may assume that 
$$
a_1\le a_2\le \cdots \le a_8 {\rm ~and~} \gcd(a_1, \ldots, a_8) =1.
$$
In our work, we assume that $a_1, \ldots, a_8 \in \{1,2,4,8\}$. For the octonary quadratic forms given by \eqref{octo}, the number of representations 
is denoted (in the above notation) by $N(1^i, 2^j, 4^k, 8^l;n)$, $i+j+k+l=8$. In our earlier paper \cite{octonary-ijnt}, we had listed some of the 
basic results in the theory of modular forms of integral and half-integral weight, which will be used in our proof. For more details we refer to 
\cite{{koblitz}, {miyake}, {shimura}}. 

We now list the main results of this paper. 

\begin{thm}\label{thm1}
For $n\in {\mathbb N}$, we have 
\begin{eqnarray}
R_{1,2}(n) &= & 6\sigma(n) - 12 \sigma(n/2)+18\sigma(n/3)-36\sigma(n/6)\\
R_{1,3}(n) &= & 3\sigma(n) - 27 \sigma(n/9)+3 k(n) \\
R_{1,4}(n) &= & 6\sigma(n)- 18 \sigma(n/2) - 18 \sigma(n/3) + 24 \sigma(n/4) + 54 \sigma(n/6) - 72 \sigma(n/12) \\
R_{1,5}(n) &= & \frac{3}{2}\sigma(n) + \frac{9}{2}\sigma(n/3) - \frac{15}{2} \sigma(n/5) - \frac{45}{2} \sigma(n/15) + \frac{9}{2}
\tau_{2,15}(n)\\
R_{2,2}(n) &= & \frac{4}{3}\sigma(n) + \frac{8}{3}\sigma(n/2) - \frac{28}{3}\sigma(n/7) - \frac{56}{3}\sigma(n/14) + \frac{2}{3}
\tau_{2,14}(n)\\
R_{2,3}(n) &= & \frac{63}{40}\sigma(n) - \frac{9}{2}\sigma(n/3) + \frac{21}{2}\sigma(n/7) - \frac{1323}{40}\sigma(n/21) + \frac{1}{2} \tau_{2,21}(n)\\
R_{2,4}(n) &= & \frac{2}{3}\sigma(n) + \frac{2}{3}\sigma(n/2) + \frac{8}{3}\sigma(n/4) - \frac{14}{3}\sigma(n/7) - \frac{14}{3}\sigma(n/14) 
- \frac{56}{3}\sigma(n/28) \nonumber \\
& & + \frac{4}{3}\tau_{2,14}(n) + \frac{8}{3}\tau_{2,14}(n/2)\\
R_{3,2}(n) &= & 2 \sigma(n) - 4\sigma(n/2) + 22\sigma(n/11) - 44 \sigma(n/22)\\
R_{3,3}(n) &= &\frac{3}{5}\sigma(n) + \frac{9}{5}\sigma(n/3) - \frac{33}{5}\sigma(n/11) - \frac{99}{5}\sigma(n/33) + \frac{16}{15}\tau_{2,11}(n) \nonumber \\ 
& &+\frac{16}{5}\tau_{2,11}(n/3) + \frac{1}{3}\tau_{2,33}(n)\\
R_{4,2}(n) &= &\frac{1}{2}\sigma(n) + \sigma(n/2) + \frac{3}{2}\sigma(n/3) - \frac{5}{2}\sigma(n/5) + 3\sigma(n/6) - 5\sigma(n/10) 
- \frac{15}{2}\sigma(n/15)  \nonumber\\
& & - 15\sigma(n/30) +\frac{1}{2}\tau_{2,15}(n) + \tau_{2,15}(n/2) + \tau_{2,30}(n) \\
R_{5,1}(n) &= & \frac{4}{3}\sigma(n)-\frac{76}{3}\sigma(n/19)+\frac{8}{3}\tau_{2,19}(n) \\
R_{5,2}(n) &= & \frac{6}{5}\sigma(n)-\frac{12}{5}\sigma(n/2)+\frac{114}{5}\sigma(n/19)-\frac{228}{5}\sigma(n/38)+\frac{4}{5}\tau_{2,38;2}(n).
\end{eqnarray}
\end{thm}

\smallskip

\noindent {\bf Note:} In the above theorem,  $k(n)$ denotes the $n$-th Fourier coefficient of the eta-quotient $\frac{\eta^3(z)\eta^3(9z)}{\eta^2(3z)}$ and $\tau_{k,N}(n)$ denotes the $n$-th Fourier coefficient of the normalized newform in the space $S_k(\Gamma_0(N),\chi)$. Also, if there are more than one newform, then $\tau_{k,N;j}(n)$ is the $n$-th Fourier coefficient of the $j$-th newform.

\bigskip

As mentioned in the introduction, since the formulas for $R_{1,2}(n), R_{1,4}(n)$ and $R_{3,2}(n)$ involve only the divisor function $\sigma(n)$, we 
use the convolution sums of the divisor functions to get formulas for a few more quadratic forms in eight variables, namely, the quadratic forms 
defined by ${\mathcal Q}_{a,\ell} \oplus j {\mathcal Q}_{a,\ell}$, which is denoted by ${\mathcal Q}_{a,\ell;j}$. Let $R_{a,\ell;j}(n)$ be the number of representations of $n$ by this quadratic form. In  \thmref{thm2}, we give formulas for $R_{a,\ell;j}(n)$ when $(a,\ell,j) =  (1,2,1), (1,2,2), (1,2,3), (1,2,4)$,\\ $(1,4,1), (1,4,2), (3,2,1)$. In order to get these formulas we need the convolution sums $W_{a,b}(n)$, $(a,b) = $ and $W_N(n)$, for 
$1\le N\le 24$. Here the convolution sums are defined as follows: 
\begin{equation}
W_{a,b}(n) = \sum_{ai+bj=n} \sigma(i) \sigma(j).
\end{equation}
We write $W_{1,N}(n)$ and $W_{N,1}(n)$ as $W_N(n)$. Also note that $W_{a,b}(n) = W_{b,a}(n)$. In all the above convolution sums, the indices 
used are natural numbers. The following theorem gives the representation numbers $R_{a,\ell;j}(n)$ for the above mentioned triplets $(a,\ell,j)$. 

\bigskip

\begin{thm}\label{thm2}
\begin{eqnarray}
R_{1,2;1}(n) & = & \frac{24}{5} \sigma_3(n) + \frac{96}{5} \sigma_3(n/2) + \frac{216}{5} \sigma_3(n/3) + \frac{864}{5} \sigma_3(n/6) + 
\frac{36}{5} \tau_{4,6}(n), \\
R_{1,2;2}(n) & = & \frac{12}{5} \sigma_3(n) - \frac{84}{5} \sigma_3(n/2) + \frac{108}{5} \sigma_3(n/3) + \frac{192}{5} \sigma_3(n/4) - \frac{756}{5} \sigma_3(n/6) \nonumber \\
& \quad &+ \frac{1728}{5} \sigma_3(n/12)  + \frac{18}{5} \tau_{4,6}(n) + \frac{72}{5} \tau_{4,6}(n/2),\\
R_{1,2;3}(n) & = &  \frac{2}{5} \sigma_3(n) + \frac{8}{5} \sigma_3(n/2) + \frac{76}{5} \sigma_3(n/3) + \frac{304}{5} \sigma_3(n/6) + \frac{162}{5} \sigma_3(n/9) \nonumber \\
& \quad &+ \frac{648}{5} \sigma_3(n/18)  + 6(n+1) \sigma(n) +\frac{3}{5} \tau_{4,6}(n) + \frac{27}{5} \tau_{4,6}(n/3) - 2\tau_{4,9}(n)\nonumber \\
& \quad  &- 8 \tau_{4,9}(n/2) + \frac{1}{5} c_{2,9}(n) + \frac{31}{5} c_{1,18}(n)\\
R_{1,2;4}(n) & = &  \frac{33}{40} \sigma_3(n) - \frac{93}{40} \sigma_3(n/2) + \frac{297}{40} \sigma_3(n/3) - \frac{93}{10} \sigma_3(n/4) 
- \frac{837}{40} \sigma_3(n/6)   \nonumber \\
& \quad & + \frac{264}{5} \sigma_3(n/8) - \frac{837}{10} \sigma_3(n/12) + \frac{2376}{5} \sigma_3(n/24) + (18 - \frac{27}{2}n)\sigma(n/3) 
 \nonumber \\
& \quad & + 27(1-n) \sigma(n/8) - \frac{27}{10} \tau_{4,6}(n) 
-18 \tau_{4,6}(n/2) - \frac{216}{5} \tau_{4,6}(n/4) \nonumber \\
& \quad  & + \frac{9}{8} \tau_{4,8}(n) + \frac{81}{8} \tau_{4,8}(n/3),\\
R_{1,4;1}(n) & = & \frac{6}{5} \sigma_3(n) + \frac{18}{5} \sigma_3(n/2) + \frac{54}{5} \sigma_3(n/3) + \frac{96}{5} \sigma_3(n/4) + \frac{162}{5} \sigma_3(n/6) \nonumber \\
& \quad &+ \frac{864}{5} \sigma_3(n/12)  - 36 \sigma(n/6) + \frac{54}{5} \tau_{4,6}(n) + \frac{216}{5} \tau_{4,6}(n/2),\\
R_{1,4;2}(n) & = &  \frac{3}{5} \sigma_3(n) + \frac{39}{5} \sigma_3(n/2) + \frac{27}{5} \sigma_3(n/3) + \frac{102}{5} \sigma_3(n/4) 
+ \frac{351}{5} \sigma_3(n/6)   \nonumber \\
& \quad & + \frac{1056}{5} \sigma_3(n/8) - \frac{1242}{5} \sigma_3(n/12) + \frac{864}{5} \sigma_3(n/24) + 54(4 -n)\sigma(n/2) 
 \nonumber \\
& \quad & - 18(1+n) \sigma(n/8)  -18(13 + 6n) \sigma(n/12) - 540n \sigma(n/24) - \frac{63}{10} \tau_{4,6}(n) 
\nonumber \\
 & \quad  &  - \frac{531}{5} \tau_{4,6}(n/2)  - \frac{1872}{5} \tau_{4,6}(n/4)  - \frac{9}{4} \tau_{4,8}(n) - \frac{81}{4} \tau_{4,8}(n/3) 
 + \frac{9}{40}c_{3,8}(n)  \nonumber \\
 & \quad  & - 54 \tau_{4,12}(n/2) + \frac{549}{40} c_{1,24}(n).
 \end{eqnarray}
\begin{eqnarray}
 R_{3,2;1}(n) & = & \frac{24}{61}\sigma_3(n) + \frac{96}{61}\sigma_3(n/2) + \frac{2904}{61}\sigma_3(n/11) + \frac{11616}{61}\sigma_3(n/22) + \frac{220}{61}a_1(n)\nonumber \\
& \quad & -\frac{480}{61}a_1(n/2) + \frac{1976}{61}a_2(n) - \frac{3296}{61}a_2(n/2) + \frac{6276}{61}a_3(n) -\frac{7680}{61}a_3(n/2) \nonumber \\
& \quad & + \frac{9280}{61}a_4(n) - \frac{7680}{61}a_4(n/2) + \frac{5440}{61}a_5(n).
\end{eqnarray}
\end{thm}

\smallskip

\begin{rmk}
{\rm 
As mentioned before, $\tau_{k,N}(n)$ denotes the $n$-th Fourier coefficient of the newform of weight $k$, level $N$. 
The coefficients $c_{2,9}(n)$,  $c_{1,18}(n)$ were defined in \cite[Definition 2.1]{aw2} and the coefficients $c_{3,8}(n)$,  $c_{1,24}(n)$ were defined in \cite[Definition 2.1]{aw4}. The remaining coefficients $a_j(n)$ that appear in the above formulas are defined by the equations \eqref{aj} to  \eqref{aj1}.
}
\end{rmk}

\smallskip

The next theorem gives the formulae for the octonary quadratic forms with coefficients $1,2,4$, and $8$ given in Table 2. We present them as two        statements, each statement corresponds to the two modular forms spaces ($M_4(\Gamma_0 (32))$ for Type I and $M_4(\Gamma_0(32),\chi_8)$ for Type II) that appear in Table 2 respectively.

\begin{thm}\label{thm3}
Let $n\in {\mathbb N}$ and $i,j,k,l$ be non-negative integers such  that $i+j+k+l =8$. \\
{\rm (i)} For each entry $(i,j,k,l)$ in Table {\rm 2} corresponding to the space $M_4(\Gamma_0(32))$, i.e. $j+l\equiv 0 (2)$, we have 
\begin{equation}\label{type1}
N(1^i,2^j,4^k,8^l;n) = \sum_{\alpha=1}^{16} c_\alpha C_\alpha(n),
\end{equation}
where $C_\alpha(n)$ are the Fourier coefficients of the basis elements $F_\alpha$ defined in \S {\rm 4.4} and the values of the constants $c_\alpha$ are given in Table {\rm 3}. \\
{\rm (ii)} For each entry $(i,j,k,l)$ in Table $2$ corresponding to the space $M_4(\Gamma_0(32),\chi_8)$, i.e., $j+l \equiv 1 (2)$, we have 
\begin{equation}
\label{type2}
N(1^i,2^j,4^k,8^l;n) = \sum_{\alpha=1}^{16} d_\alpha D_\alpha(n),
\end{equation}
where $D_\alpha(n)$ are the Fourier coefficients of the basis elements $G_\alpha$ defined in \S {\rm 4.5} and the values of the constants $d_\alpha$ are given in Table {\rm 4}.
\end{thm}

\smallskip

\section{Sample formulas}
In this section we shall give explicit formulas for a few cases of \eqref{type1} and \eqref{type2} in \thmref{thm3}. We first give the formulas 
for the cases (1,0,1,6) and (1,1,1,5) in  Table 2 (Type I), which correspond to the space $M_4(\Gamma_0(32)).$\\

\noindent 
For $n\in {\mathbb N}$, we have
\begin{equation*}
\begin{split}
%{\rm (i)} ~
N(1^1,4^1,8^6;n) &=\frac{1}{64}   \sigma_{3}(n)-\frac{9}{64} \sigma_{3}(n/2)+ \frac{17}{8} \sigma_{3}(n/4)-2\sigma_{3}(n/8) -16\sigma_{3}(n/16)\\
&\quad +256\sigma_{3}(n/32)+\frac{1}{64}   \sigma_{3;\chi_{-4},\chi_{-4}}(n)+\frac{31}{64} a_{4,8}(n)+ 2a_{4,8}(n/4)+\frac{31}{64} a_{4,16}(n)\\
&\quad +\frac{13}{8} a_{4,32,1}(n)+\frac{3}{4} a_{4,32,2}(n)-\frac{5}{8} a_{4,32,3}(n),  \\
%\end{split}
%\end{equation*}
%\smallskip
%\begin{equation*}
%\begin{split}
%{\rm (ii)} ~
N(1^1,2^1,4^1,8^5;n)&=\frac{1}{32} \sigma_{3}(n)-\frac{1}{32} \sigma_{3}(n/2)-16\sigma_{3}(n/16)-256\sigma_{3}(n/32)+ \frac{11}{32} a_{4,8}(n)\\
&\quad + \frac{3}{4} a_{4,8}(n/2)+2 a_{4,8}(n/4)+\frac{5}{8} a_{4,16}(n)+ \frac{11}{8} a_{4,32,1}(n)+\frac{1}{4} a_{4,32,2}(n)\\
&\quad - \frac{3}{8} a_{4,32,3}(n). 
\end{split}
\end{equation*}

\bigskip

\noindent Next we give the formulas for the cases (1,0,0,7) and (1,1,2,4) in  Table 2 (Type II), which  correspond to the space $M_4(\Gamma_0(32),\chi_8).$\\

\noindent 
For $n\in {\mathbb N}$, we have
\begin{equation*}
\begin{split}
N(1^1,8^7;n) &=\frac{1}{88}\sigma_{3,\chi_0,\chi_2}(n)-\frac{1}{88} \sigma_{3,\chi_0,\chi_2}(n/2)-\frac{2}{11} \sigma_{3,\chi_0,\chi_2}(n/4)+\frac{1}{88} \sigma_{3,\chi_2,\chi_0}(n)\\
&\quad -\frac{1}{11} \sigma_{3,\chi_2,\chi_0}(n/2)-\frac{16}{11} \sigma_{3,\chi_2,\chi_0}(n/4)+\frac{1}{88}\sigma_{3;\chi_{-4},\chi_{-8}}(n)+ \frac{1}{88}\sigma_{3;\chi_{-8},\chi_{-4}}(n)\\
&\quad +\frac{43}{176} a_{4,8,\chi_8;1}(n)+\frac{43}{22} a_{4,8,\chi_8;1}(n/2)+\frac{8}{11} a_{4,8,\chi_8;1}(n/4)-\frac{129}{176} a_{4,8,\chi_8;2}(n)\\
&\quad -\frac{43}{44} a_{4,8,\chi_8;2}(n/2)- \frac{4}{11}  a_{4,8,\chi_8;2}(n/4)+\frac{43}{44} a_{4,32,\chi_8;1}(n)+\frac{43}{44}a_{4,32,\chi_8;2}(n),\\
%\end{split}
%\end{equation*}
%\smallskip
%\begin{equation*}
%\begin{split}
N(1^1,2^1,4^2,8^4;n)&=\frac{2}{11} \sigma_{3,\chi_0,\chi_2}(n/4)+\frac{1}{22} \sigma_{3,\chi_2,\chi_0}(n) +\frac{3}{22} a_{4,8,\chi_8;1}(n)+2 a_{4,8,\chi_8;1}(n/2) \\
&\quad -\frac{48}{11} a_{4,8,\chi_8;1}(n/4)-\frac{9}{11} a_{4,8,\chi_8;2}(n) + a_{4,8,\chi_8;2}(n/2)+ \frac{16}{11} a_{4,8,\chi_8;2}(n/4)\\
&\quad + a_{4,32,\chi_8;1}(n)+2 a_{4,32,\chi_8;2}(n).\\
\end{split}
\end{equation*}

\section{Proofs of Theorems}

\subsection{Proof of \thmref{thm1}}

Let $\Theta_{a,\ell}(z)$ denote the theta series associated to the quadratic form ${\mathcal Q}_{a,\ell}$. Then 
\begin{equation}
\Theta_{a,\ell}(z) = \Theta_a(z) \Theta_a(\ell z),
\end{equation}
where $\Theta_a(z)$ is the theta function associated to the quadratic form ${\mathcal Q}_a$. i.e., 
\begin{equation}
\Theta_a(z) = \sum_{m,n=-\infty}^{\infty}  q^{m^2+mn+a n^2}.
\end{equation}
Recall $q = e^{2\pi iz}$. Since $R_{a, \ell}(n)$ is the number of representations of a positive integer $n$ by the quadratic form ${\mathcal Q}_{a,\ell}$, 
we see that 
\begin{equation}\label{repn}
\Theta_{a,\ell}(z) = 1 + \sum_{n=1}^\infty R_{a,\ell}(n) q^n.
\end{equation}
So, it is sufficient to write the theta series $\Theta_{a,\ell}(z)$ in terms of a basis of the space of modular forms in order to get our formulas.

\begin{lem}
The theta series $\Theta_{a,\ell}(z)$ is a modular form of weight $2$ on $\Gamma_0({\rm lcm}[\ell,(4a -1)])$ with trivial character. 
\end{lem}

\begin{proof}
By \cite[Theorem 4]{Schoeneberg}, it follows that $\Theta_a(z)$ is a modular form of weight $1$ on $\Gamma_0(4a-1)$ with character 
$\left(\frac{\cdot}{4a-1}\right)$. Also, it is a well-known fact that if $f$ is a modular form of integer weight $k$ on $\Gamma_0(N)$ with character 
$\psi$, then for a positive integer $d$, the function $f(dz)$ is a modular form of same weight $k$ on $\Gamma_0(dN)$ with character $\psi$. 
Further, if $f_i$ are modular forms of weight $k_i$, on $\Gamma_0(N_i)$ with chracter $\psi_i$, $i=1,2$, then the product $f_1f_2$ is  a modular form of weight $k_1+k_2$ on $\Gamma_0({\rm lcm}[N_1,N_2])$ with character $\psi_1\psi_2$. For these facts, we refer to \cite[Chapter 3]{koblitz}. 
We also refer to the proof of Fact II in our earlier work \cite{octonary-ijnt}, which contains details of the above arguments. 
Therefore, $\Theta_{a,\ell}(z)$ is a modular form of weight $2$ on $\Gamma_0({\rm lcm}[\ell,(4a-1)])$. 
\end{proof}

Let $(a,\ell)$ be an element of $\textbf{A}$. Consider the quadratic form ${\mathcal Q}_{a,\ell}$. By the above lemma, the corresponding 
theta series $\Theta_{a,\ell}(z)$ is a modular form in the space $M_2({\Gamma_0(\rm lcm}[\ell, (4a-1)]))$. Let us assume that the dimension of this vector space is $d_{a,\ell}$. If $\{f_i : 1\le i\le d_{a,\ell}\}$ is a basis of $M_2(\Gamma_0({\rm lcm}[\ell, (4a-1)]))$, then we can write the theta 
series $\Theta_{a,\ell}(z)$ in terms of this basis. So, let 
$$
\Theta_{a,\ell}(z) = \sum_{i=1}^{d_{a,\ell}} c_i f_i(z).
$$
Combining this with \eqref{repn} and comparing the $n$-th Fourier coefficients, we obtain the required formulas for $R_{a,\ell}(n)$. 

We shall give below a basis of the modular forms space used in our formulas corresponding to each pair $(a,\ell)$ in the set $\textbf{A}$. Using these bases, the formulas mentioned in \thmref{thm1} follow by comparing the $n$-th Fourier coefficients as demonstrated above. We shall be using the 
notation \eqref{eta-q} for the eta-quotients. 

\smallskip

Before we proceed, we define certain modular form of weight $2$ using the quasimodular form $E_2(z)$. 
For natural numbers  $a, b$ with $a\vert b$, $a\not= b$, define the function $\Phi_{a,b}(z)$ by 
\begin{equation}\label{Phi}
\Phi_{a,b}(z) = \frac{1}{b-a} (b E_2(bz) - a E_2(az)).  
\end{equation}
Using the transformation properties of $E_2(z)$, it follows that $\Phi_{a,b}(z)$ is a modular form belonging to the space $M_2(\Gamma_0(b))$. 
We shall use these type of forms to construct  our bases for the spaces of modular forms of weight $2$.

\smallskip 

\noindent {\bf A basis for  the space $M_2(\Gamma_0(6))$ (the case $(a,\ell) = (1,2)$)}: The space $M_2(\Gamma_0(6))$ is $3$ dimensional 
and a basis is given by 
$$
\left\{\Phi_{1,2}(z),\Phi_{1,3}(z),\Phi_{1,6}(z)\right\}.
$$
Therefore, 
\begin{equation*}
\begin{split}
\Theta_{1,2}(z) & = \frac{1}{4}\Phi_{1,2}(z) - \frac{1}{2}\Phi_{1,3}(z)+\frac{5}{4}\Phi_{1,6}(z)\\
& = -\frac{1}{4}E_2(z)+\frac{1}{2}E_2(2z) - \frac{3}{4}E_2(3z)+\frac{3}{2}E_2(6z),
\end{split}
\end{equation*}
from which the formula for $R_{1,2}(n)$ follows. 

\smallskip

\noindent {\bf A basis for  the space $M_2(\Gamma_0(9))$ (the case $(a,\ell) = (1,3)$)}: The vector space $M_2(\Gamma_0(9))$ has dimension $3$ and 
we use the following basis.
$$
\left\{\Phi_{1,3}(z),\Phi_{1,9}(z),\Psi_{2,9}(z)\right\},
$$
where $\Psi_{2,9}(z)$ is the eta-quotient $\Psi_{2,9}(z ) = \displaystyle{\frac{\eta^3(z)\eta^3(9z)}{\eta^2(3z)}}$. 
We have 
\begin{equation*}
\begin{split}
\Theta_{1,3}(z) & = \Phi_{1,9}(z) + 3 \Psi_{2,9}(z)\\
& = -\frac{1}{8}E_2(z) + \frac{9}{8} E_2(9z) + 3\Psi_{2,9}(z).
\end{split}
\end{equation*}

\smallskip

\noindent {\bf A basis for  the space $M_2(\Gamma_0(12))$ (the case $(a,\ell) = (1,4)$)}: 
A basis for the space $M_2(\Gamma_0(12))$ (which has dimension $5$)  is given by 
$$
\left\{\Phi_{1,2}(z),\Phi_{1,3}(z),\Phi_{1,4}(z),\Phi_{1,6}(z),\Phi_{1,12}(z)\right\}.
$$
So, $\Theta_{1,4}(z)$ can be written as 
\begin{equation*}
\begin{split}
\Theta_{1,4}(z) &  = \frac{3}{8}\Phi_{1,2}(z)+\frac{1}{2}\Phi_{1,3}(z)-\frac{3}{4}\Phi_{1,4}(z)-\frac{15}{8}\Phi_{1,6}(z)+\frac{11}{4}\Phi_{1,12}(z)\\
& = -\frac{1}{4}E_2(z)+\frac{3}{4}E_2(2z)+\frac{3}{4}E_2(3z)-E_2(4z)-\frac{9}{4}E_2(6z)+3E_2(12z).
\end{split}
\end{equation*}

\smallskip

\noindent {\bf A basis for  the space $M_2(\Gamma_0(15))$ (the case $(a,\ell) = (1,5)$)}: 
%\subsection{A basis for $M_2(15)$ and proof of theorem [2.12]}
The vector space $M_2(\Gamma_0(15))$ has dimension $4$ and the subspace of cusp forms $S_2(\Gamma_0(15))$ is one dimensional. Let $\Delta_{2,15}(z)$ be the 
unique normalized newform in the space $S_2(\Gamma_0(15))$, which is given by an eta-quotient and we put 
\begin{equation}\label{2-15}
\Delta_{2,15}(z) = 1^1  3^1  5^1 15^1 = \sum_{n\ge 1} \tau_{2,15}(n) q^n.
\end{equation}
We consider the following basis for $M_2(\Gamma_0(15))$:
$$
\left\{\Phi_{1,3}(z),\Phi_{1,5}(z),\Phi_{1,15}(z),\Delta_{2,15}(z)\right\}.
$$
In this case, we have 
\begin{equation*}
\begin{split}
\Theta_{1,5}(z) & = -\frac{1}{8}\Phi_{1,3}(z) + \frac{1}{4}\Phi_{1,5}(z) + \frac{7}{8}\Phi_{1,15}(z) + \frac{9}{2}\Delta_{2,15}(z) \\
& =-\frac{1}{16}E_2(z) - \frac{3}{16}E_2(3z) + \frac{5}{16}E_2(5z) + \frac{15}{16}E_2(15z) + \frac{9}{2}\Delta_{2,15}(z).
\end{split}
\end{equation*}

\smallskip

\noindent {\bf A basis for  the space $M_2(\Gamma_0(14))$ (the case $(a,\ell) = (2,2)$)}: 
A basis for the $4$ dimensional vector space $M_2(\Gamma_0(14))$ is given by 
$$
\left\{\Phi_{1,2}(z),\Phi_{1,7}(z),\Phi_{1,14}(z),\Delta_{2,14}(z)\right\},
$$
where  $\Delta_{2,14}(z)$ is the unique normalized newform in $S_2(\Gamma_0(14))$, which is given by  
\begin{equation}\label{2-14}
\Delta_{2,14}(z) = 1^{1} 2^{1} 7^{1} 14^{1} = \sum_{n\ge 1} \tau_{2,14}(n) q^n.
\end{equation}
With this basis, the theta series $\Theta_{2,2}(z)$ has the following expression.
\begin{equation*}
\begin{split}
\Theta_{2,2}(z) & = -\frac{1}{18}\Phi_{1,2}(z) + \frac{1}{3}\Phi_{1,7}(z) + \frac{13}{18}\Phi_{1,14}(z) + \frac{2}{3}\Delta_{2,14}(z) \\
& = -\frac{1}{18}E_2(z) - \frac{1}{9}E_2(2z) + \frac{7}{18}E_2(7z) + \frac{7}{9}E_2(14z) + \frac{2}{3}\Delta_{2,14}(z).
\end{split}
\end{equation*}
\smallskip

\noindent {\bf A basis for  the space $M_2(\Gamma_0(21))$ (the case $(a,\ell) = (2,3)$)}: 
Let $\Delta_{2,21}(z)$ be the unique normalized newform in $S_2(\Gamma_0(21))$, which is given by the following eta-quotient:
\begin{equation}\label{2-21}
\begin{split}
\Delta_{2,21}(z)\!\!& =\!\! \frac{\eta(7z)}{2 \eta^2(z) \eta(3z)\eta(9z)\eta(21z)}(3\eta^2(z)\eta^2(7z)\eta^4(9z)-\eta^5(3z) \eta(7z)\eta(9z)\eta(21z)+ 3 \eta^4(z) \eta^2(9z) \eta^2(63z)\\
& \quad + 7 \eta(z) \eta^2(3z) \eta(9z) \eta^4(21z)+ 3 \eta^3(z) \eta(7z) \eta^3(9z) \eta(63z) - 3\eta(z) \eta^5(3z) \eta(21z) \eta(63z)).
\end{split}
\end{equation}
Now a basis for this space is given by 
$$
\left\{\Phi_{1,3}(z),\Phi_{1,7}(z),\Phi_{1,21}(z),\Delta_{2,21}(z)\right\}.
$$
We give the expression for the corresponding theta series.
\begin{equation*}
\begin{split}
\Theta_{2,3}(z) &= \frac{1}{8}\Phi_{1,3}(z) - \frac{3}{8}\Phi_{1,7}(z) + \frac{21}{16}\Phi_{1,21}(z) + \frac{1}{2}\Delta_{2,21}(z) \\
& = -\frac{21}{320}E_2(z) + \frac{3}{16}E_2(3z) - \frac{7}{16}E_2(7z) + \frac{441}{320}E_2(21z) + \frac{1}{2}\Delta_{2,21}(z).
\end{split}
\end{equation*}

\smallskip

\noindent {\bf A basis for  the space $M_2(\Gamma_0(28))$ (the case $(a,\ell) = (2,4)$)}: 
In this case, the cusp forms space $S_2(\Gamma_0(28))$ is spanned by $\Delta_{2,14}(z)$ and $\Delta_{2,14}(2z)$ and we use the following basis:
$$
\left\{\Phi_{1,2}(z),\Phi_{1,4}(z),\Phi_{1,7}(z),\Phi_{1,14}(z),\Phi_{1,28}(z),\Delta_{2,14}(z),\Delta_{2,14}(2z)\right\}.
$$
The newform $\Delta_{2,14}(z)$ is given by \eqref{2-14}. We give the expression for the theta series.
\begin{equation*}
\begin{split}
\Theta_{2,4}(z) &= -\frac{1}{72}\Phi_{1,2}(z)-\frac{1}{12}\Phi_{1,4}(z)+\frac{1}{6}\Phi_{1,7}(z)+\frac{13}{72}\Phi_{1,14}(z)+\frac{3}{4}\Phi_{1,28}(z)+\frac{4}{3}\Delta_{2,14}(z)+\frac{8}{3}\Delta_{2,14}(2z)\\
& = -\frac{1}{36}E_2(z)-\frac{1}{36}E_2(2z)-\frac{1}{9}E_2(4z)+\frac{7}{36}E_2(7z)+\frac{7}{36}E_2(14z) +\frac{7}{9}E_2(28z)\\
& \quad +\frac{4}{3}\Delta_{2,14}(z)+\frac{8}{3}\Delta_{2,14}(2z).
\end{split}
\end{equation*}

\smallskip

\noindent {\bf A basis for  the space $M_2(\Gamma_0(22))$ (the case $(a,\ell) = (3,2)$)}: First we give the newform of weight $2$ on $\Gamma_0(11)$. 
\begin{equation}\label{2-11}
\Delta_{2,11}(z) =  1^2  11^2 = \sum_{n\ge 1} \tau_{2,11}(n) q^n. 
\end{equation}
For getting the required formula, we use the following basis:
$$
\left\{\Phi_{1,2}(z),\Phi_{1,11}(z),\Phi_{1,22}(z),\Delta_{2,11}(z),\Delta_{2,11}(2z)\right\}.
$$
The expression for the theta series $\Theta_{3,2}(z)$ is given below. 
\begin{equation*}
\begin{split}
\Theta_{3,2}(z) & = \frac{1}{12}\Phi_{1,2}(z) - \frac{5}{6}\Phi_{1,11}(z) + \frac{7}{4}\Phi_{1,22}(z) \\
& = -\frac{1}{12}E_2(z) + \frac{1}{6}E_2(2z) - \frac{11}{12}E_2(11z) + \frac{11}{6}E_2(22z).
\end{split}
\end{equation*}

\smallskip

\noindent {\bf A basis for  the space $M_2(\Gamma_0(33))$ (the case $(a,\ell) = (3,3)$)}: 
In this case the dimension of the space is $6$. We need the newform of level $33$. Since explicit expression of this newform is not 
known, we give below its first few Fourier coefficients (using SAGE).
\begin{equation}\label{2-33}
\Delta_{2,33}(z) = q + q^2 - q^3 - q^4 - 2 q^5 - q^6 + 4 q^7 - 3 q^8 + q^9 - 2 q^{10} + {O}(q^{11})
% q^11 + q^12 - 2*q^13 + 4*q^14 + 2*q^15 - q^16 - 2*q^17 + q^18 + O(q^20)
\end{equation} 
We use the following basis for $M_2(\Gamma_0(33))$:
$$
\left\{\Phi_{1,3}(z),\Phi_{1,11}(z),\Phi_{1,33}(z),\Delta_{2,11}(z),\Delta_{2,11}(3z),\Delta_{2,33}(z)\right\}.
$$
Using this basis, we have 
\begin{equation*}
\begin{split}
\Theta_{3,3}(z) & = -\frac{1}{20}\Phi_{1,3}(z) + \frac{1}{4}\Phi_{1,11}(z) + \frac{4}{5}\Phi_{1,33}(z) + \frac{16}{15}\Delta_{2,11}(z) + \frac{16}{5}\Delta_{2,11}(3z) + \frac{1}{3}\Delta_{2,33}(z) \\
& = -\frac{1}{40}E_2(z) - \frac{3}{40}E_2(3z) + \frac{11}{40}E_2(11z) + \frac{33}{40}E_2(33z) + \frac{16}{15}\Delta_{2,11}(z)
+ \frac{16}{5}\Delta_{2,11}(3z) + \frac{1}{3}\Delta_{2,33}(z).
\end{split}
\end{equation*}
\smallskip

\noindent {\bf A basis for  the space $M_2(\Gamma_0(30))$ (the case $(a,\ell) = (4,2)$)}: The normalized newform of level $15$ is given by \eqref{2-15}. For level $30$ it is defined below.
\begin{equation}\label{2-30}
\Delta_{2,30}(z) = 3^1  5^1  6^1 10^1 - 1^1 2^1 15^1 30^1= \displaystyle{\sum_{n\ge 1}} \tau_{2,30}(n) q^n.
\end{equation}
Following is a basis for the space $M_2(\Gamma_0(30))$.
$$
\left\{\Phi_{1,2}(z),\Phi_{1,3}(z),\Phi_{1,5}(z),\Phi_{1,6}(z),\Phi_{1,10}(z),\Phi_{1,15}(z),\Phi_{1,30}(z),\Delta_{2,15}(z),\Delta_{2,15}(2z),\Delta_{2,30}(z)\right\}.
$$
Using the above basis, we have 

\begin{equation*}
\begin{split}
\Theta_{4,2}(z) & = -\frac{1}{48}\Phi_{1,2}(z) - \frac{1}{24}\Phi_{1,3}(z) + \frac{1}{12}\Phi_{1,5}(z) - \frac{5}{48}\Phi_{1,6}(z)
+ \frac{3}{16}\Phi_{1,10}(z) + \frac{7}{24}\Phi_{1,15}(z) \\
& \quad + \frac{29}{48}\Phi_{1,30}(z) + \frac{1}{2}\Delta_{2,15}(z) + \Delta_{2,15}(2z)+\Delta_{2,30}(z)\\
& = - \frac{1}{48}E_2(z) -\frac{1}{24} E_2(2z) - \frac{1}{16}E_2(3z) + \frac{5}{48}E_2(5z) - \frac{1}{8}E_2(6z) + \frac{5}{24}E_2(10z) + \frac{5}{16}E_2(15z) \\
& \quad + \frac{5}{8}E_2(30z) + \frac{1}{2}\Delta_{2,15}(z) + \Delta_{2,15}(2z)+\Delta_{2,30}(z).
\end{split}
\end{equation*}

\smallskip

\noindent {\bf A basis for  the space $M_2(\Gamma_0(19))$ (the case $(a,\ell) = (5,1)$)}: 
For defining the newform of level $19$, we use the Ramanujan theta functions  $\Phi(z)$ and $\Psi(z)$ which are defined below.
\begin{equation}
\begin{split}
\Phi(z)& := \frac{\eta^5(2z)}{\eta^2(z)\eta^2(4z)},\\ 
\Psi(z)& := q^{-1/8}\frac{\eta^2(2z)}{\eta(z)}.
\end{split}
\end{equation}
We give the newform $\Delta_{2,19}(z)$ as follows.
\begin{equation}\label{2-19}
\Delta_{2,19}(z)= q\left\{\Psi(4z)\Phi(38z)-q^2\Psi(z)\Psi(19z)+q^9\Phi(2z)\Psi(76z)\right\}^2 := \displaystyle{\sum_{n\ge 1}} \tau_{2,19}(n) q^n
\end{equation}
The vector space $M_2(\Gamma_0(19))$ is spanned by the following two modular forms:
$$
\left\{\Phi_{1,19}(z),\Delta_{2,19}(z)\right\}
$$
Now we give the expression for the corresponding theta function. 
\begin{equation*}
\begin{split}
\Theta_{5,1}(z) & = \Phi_{1,19}(z)+\frac{8}{3}\Delta_{2,19}(z) \\
& = -\frac{1}{18}E_2(z) + \frac{19}{18}E_2(19z) + \frac{8}{3}\Delta_{2,19}(z).
\end{split}
\end{equation*}

\smallskip

\noindent {\bf A basis for  the space $M_2(\Gamma_0(38))$ (the case $(a,\ell) = (5,2)$)}: In this case we need two newforms of level $38$. Explicit expression of these newforms are not known. However, using SAGE one can get their Fourier expansion (with certain number of Fourier coefficients) which 
we give below. 
\begin{equation}
\begin{split}
\Delta_{2,38;1}(z)& = q - q^2 + q^3 + q^4 - q^6 - q^7 - q^8 - 2q^9 + O(q^{10}) =  \sum_{n\ge 1} \tau_{2,38;1}(n) q^n, \\
\Delta_{2,38;2}(z)& = q + q^2 - q^3 + q^4 - 4q^5 - q^6 + 3q^7 + q^8 - 2q^9 + O(q^{10}) = \sum_{n\ge 1} \tau_{2,38;2}(n) q^n .
\end{split}
\end{equation}
A basis for the space $M_2(\Gamma_0(38))$ is given by
$$
\left\{\Phi_{1,2}(z),\Phi_{1,19}(z),\Phi_{1,5}(z),\Phi_{1,38}(z),\Delta_{2,19}(z),\Delta_{2,19}(2z),\Delta_{2,38;1}(z),\Delta_{2,38;2}(z)\right\}.
$$
In this case, the theta series has the following expression.
\begin{equation*}
\begin{split}
\Theta_{5,2}(z) & = \frac{1}{20}\Phi_{1,2}(z) - \frac{9}{10}\Phi_{1,19}(z) + \frac{37}{20}\Phi_{1,5}(z) + \frac{4}{5}\Delta_{2,38;2}(z) \\
& = -\frac{1}{20}E_2(z) + \frac{1}{10}E_2(2z) - \frac{19}{20}E_2(19z) + \frac{19}{10}E_2(38z) + \frac{4}{5}\Delta_{2,38;2}(z).
\end{split}
\end{equation*}
Proof of \thmref{thm1} is now complete.

\bigskip

\subsection{Proof of \thmref{thm2}}
We shall demonstrate the method by giving a proof of the formula for $R_{1,2;j}(n)$, $1\le j\le4$. The rest of the proofs are similar. 
It is clear that 
\begin{equation*}
%\begin{split}
R_{1,2;j}(n)  = \sum_{a, b \in {\mathbb N}_0\atop{a+bj =n}} R_{1,2}(a) R_{1,2}(b). 
%\end{split}
\end{equation*}
Now using the formula for $R_{1,2}(n)$ from \thmref{thm1} with the convention $R_{1,2}(0) =1$, we get 
\begin{equation*}
\begin{split}
R_{1,2;j}(n) & = R_{1,2}(n) + R_{1,2}(n/j) + \sum_{a,b \in {\mathbb N}\atop {a+bj=n}} R_{1,2}(a) R_{1,2}(b)\\
& = R_{1,2}(n)  + R_{1,2}(n/j) + \sum_{a,b\in {\mathbb N}\atop{a+bj=n}} (6 \sigma(a) - 12 \sigma(a/2) + 18 \sigma(a/3) -36 \sigma(a/6)) \\
& \hskip 5cm  (6 \sigma(b) - 12 \sigma(b/2) + 18 \sigma(b/3) -36 \sigma(b/6))\\
& =  R_{1,2}(n)  + R_{1,2}(n/j) + 36 W_j(n) - 72 W_{2j}(n) + 108 W_{3j}(n) - 216 W_{6j}(n) - 72 W_{2,j}(n)\\ 
& \quad + 108  W_{3,j}(n) - 216 W_{6,j}(n) - 216 W_{2,3j}(n) - 216 W_{3,2j}(n) + 144 W_j(n/2) +324 W_j(n/3) \\
& \quad + 1296 W_j(n/6)  - 648 W_{2j}(n/3) + 432 W_{3j}(n/2) - 648 W_{2,j}(n/3)  + 432 W_{3,j}(n/2).
\end{split}
\end{equation*}
We now use the convolution sums $W_{a,b}(n)$ and $W_N(n)$ obtained by several authors (see the table below) in the last step and get the required formulas for $R_{1,2;j}(n)$ for $1\le j\le 4$. 

\bigskip

\begin{center}
\begin{tabular}{|l|l|l|c|}
\hline 
$(a,\ell;j)$ & \quad Convolution sums & \quad Convolution sums& References \\
& \qquad \quad $W_N(n)$ & \quad  \qquad  $W_{a,b}(n)$ & \\
\hline
$(1,2;1)$ & $W_N(n)$, $N=1,2,3,6$& $W_{2,3}(n)$ & \cite{{aw3}, {royer}}\\
\hline
$(1,2;2)$ & $W_N(n)$, $N=1,2,3,4,6,12$& $W_{2,3}(n)$, $W_{3,4}(n)$ & \cite{{aw3}, {okayama}, {royer}}\\
\hline
$(1,2;3)$ & $W_N(n)$, $N=1,2,3,6,9,18$& $W_{2,3}(n)$, $W_{2,9}(n)$ & \cite{{aw2}, {aw3}, {royer}}\\
\hline
$(1,2;4)$ & $W_N(n)$, $N=2,4,6,8,12,24$& $W_{2,3}(n)$, $W_{3,4}(n)$, $W_{3,8}(n)$ & \cite{{aw4}, {aw3}, {okayama}, {royer}}\\
\hline
$(1,4;1)$ & $W_N(n)$, $N=1,2,3,4,6,12$& $W_{2,3}(n)$, $W_{3,4}(n)$ & \cite{{aw3}, {okayama}, {royer}}\\
\hline
$(1,4;2)$ & $W_N(n)$, $N=1,2,4,6,8,12,24$& $W_{2,3}(n)$, $W_{3,4}(n)$, $W_{3,8}(n)$ & \cite{{aw4}, {aw3}, {okayama},{royer}}\\
\hline
$(3,2;1)$ & $W_N(n)$, $N=1,2,22$& $W_{2,11}(n)$ & \cite{{ntienjem}, {royer}}\\
\hline 
\end{tabular}
\end{center}

\smallskip

To get the formula for $R_{3,2;1}(n)$ we also need the convolution sum $W_{11}(n)$. Though this convolution sum is obtained by E. Royer in \cite{royer}, it involved a pair of terms with complex coefficients. In order to avoid this expression, we compute below the convolution sum 
$W_{11}(n)$ which involves only the rational coefficients. 

\smallskip

\noindent {\bf The convolution sum $W_{11}(n)$}:  First  we compute an explicit basis for the space $M_4(\Gamma_0(22))$. The dimension of this vector space is $11$ and the cuspidal dimension is $7$. The following $7$ eta-quotients form a basis for the space of cusp forms $S_4(\Gamma_0(22))$.

\begin{eqnarray}\label{aj}
A_{1}(z) & = & 1^{6}2^{-2} 11^{6}  22^ {-2} := \displaystyle{\sum_{n\ge 1}} a_{1}(n) q^n,\\
%\eta^2(z) \eta^2(2z) \eta^2(3z) \eta^2(6z) 
 A_{2}(z) & = & 1^{4}11^{4}:= \displaystyle{\sum_{n\ge 1}} a_{2}(n) q^n, \\
%\eta^4(2z) \eta^4(4z) := 
 A_{3}(z) &= &1^{2}  2^{2} 11^{2}  22^{2}:= \displaystyle{\sum_{n\ge 1}} a_{3}(n) q^n, 
 \end{eqnarray}
 \begin{eqnarray}
 A_{4}(z)  &=  &2^{4} 22^{4} := \displaystyle{\sum_{n\ge 1}} a_{4}(n) q^n,\\
%\frac{\eta^{16}(4z)}{\eta^4(2z)\eta^4(8z)} 
 A_{5}(z)  &=  &1^{-2}  2^{6} 11^{-2}  22^{6} := \displaystyle{\sum_{n\ge 1}} a_{5}(n) q^n,\\
 A_{6}(z)  &= & 1^{-1}  2^{1} 11^{3}  22^5 := \displaystyle{\sum_{n\ge 1}} a_{6}(n) q^n,\\
A_{7}(z) &=&  1^{-5}  2^{9} 11^{7}  22^{-3}:= \displaystyle{\sum_{n\ge 1}} a_{7}(n) q^n.\label{aj1}
\end{eqnarray}

\smallskip

By taking a basis of the Eisenstein series for the space $M_4(\Gamma_0(22))$ as $\left\{E_4(tz) : t\vert 22\right\}$, we get the following full 
basis for $M_4(\Gamma_0(22))$.
$$
\left\{E_4(tz), A_j(z) : t\vert 22, 1\le j\le 7\right\}.
$$
In order to get the convolution sum $W_{11}(n)$, we express the modular form of weight $4$ given by $(E_2(z)-11E_2(11z))^2$ in 
terms of the above basis. So, we get the following expression.
\begin{equation*}
\begin{split}
(E_2(z)-11E_2(11z))^2 & = \frac{50}{61}E_4(z) + \frac{6050}{61}E_4(11z) + \frac{17280}{61} A_1(z) + \frac{118656}{61}A_2(z) \\
& \quad + \frac{276480}{61}A_3(z) + \frac{276480}{61}A_4(z).
\end{split}
\end{equation*}

Now, by comparing the $n$-th coefficient on both the sides, we get the expression for $W_{11}(n)$ as 
\begin{equation}\label{w11}
\begin{split}
W_{11}(n)& = \frac{5}{1464}\sigma_3(n) + \frac{605}{1464}\sigma_3(n/11) + \left(\frac{1}{21} - \frac{n}{44}\right)\sigma(n) + \left( \frac{1}{21}-\frac{n}{4}\right)\sigma(n/11)\\
& \quad -\frac{15}{671}a_1(n) - \frac{103}{671}a_2(n) - \frac{240}{671} a_3(n) -  \frac{240}{671} a_4(n).
\end{split}
\end{equation}

\smallskip

\subsection{Proof of \thmref{thm3}}
We observe that the theta series corresponding to the quadratic form given by \eqref{octo} is the following product:
$$
\Theta^i(z) \Theta^j(2z) \Theta^k(4z)\Theta^l(8z).
$$
Therefore, by Fact II of \cite{octonary-ijnt}, all of them belong to the space of modular forms of weight $4$ on $\Gamma_0(32)$ with character 
depending on the parity of $j+l$. When $j+l$ is even, then the above product of theta series belongs to $M_4(\Gamma_0(32))$ and if $j+l$ is odd, 
then it belongs to $M_4(\Gamma_0(32), \chi_8)$. Therefore, as in the proof of \thmref{thm1}, the essence of the proof lies in giving explicit bases 
for these vector spaces. 
%\smallskip

\subsection{A basis for $M_4(\Gamma_0(32))$ and proof of \thmref{thm3}(i).}

The vector space $M_4(\Gamma_0(32))$ has dimension $16$ and the space of Eisenstein series has dimension $8$. So, 
$\dim_{\mathbb C}S_4(\Gamma_0(32)) = 8$. For $d=8$ and $16$, $S_4^{new}(\Gamma_0(d))$ is one-dimensional and 
$\dim_{\mathbb C}S_4^{new}(\Gamma_0(32)) =3$.\\
Let us define some eta-quotients and use them to give an explicit basis for $S_4(\Gamma_0(32))$.
Let
\begin{eqnarray}
%\smallskip
%\smallskip
&f_{4,8}(z) = 2^4  4^4 := \displaystyle{\sum_{n\ge 1}} a_{4,8}(n) q^n, \\
%\eta^4(2z) \eta^4(4z) := 
&f_{4,16}(z) = 2^{-4} 4^{16} 8^{-4} := \displaystyle{\sum_{n\ge 1}} a_{4,16}(n) q^n, ~
%\frac{\eta^{16}(4z)}{\eta^4(2z)\eta^4(8z)} 
\end{eqnarray}
\begin{eqnarray}
&g_{4,32,1}(z) =  1^{-2} 2^{1} 4^{8} 8^{3} 16^{-2}:= \displaystyle{\sum_{n\ge 1}} a_{4,32,1}(n) q^n,\\
%\frac{\eta^{11}(2z)\eta^{11}(6z)}{\eta^4(z)\eta^4(3z)\eta^3(4z)\eta^3(12z)} \\ 
&g_{4,32,2}(z) =  1^{2} 2^{3} 8^{1} 16^{2}  := \displaystyle{\sum_{n\ge 1}} a_{4,32,2}(n) q^n,\\
&g_{4,32,3}(z) = 1^{-4} 2^{6} 4^{8} 8^{-2} := \displaystyle{\sum_{n\ge 1}} a_{4,32,3}(n) q^n,\\
&f_{4,32,1}(z) := \displaystyle{\sum_{n\equiv1(2)}} a_{4,32,1}(n) q^n,\\
%\frac{\eta^{11}(2z)\eta^{11}(6z)}{\eta^4(z)\eta^4(3z)\eta^3(4z)\eta^3(12z)} \\ 
&f_{4,32,2}(z) :=\displaystyle{\sum_{n\equiv1(2)}} a_{4,32,2}(n) q^n,\\
&f_{4,32,3}(z) :=\displaystyle{\sum_{n\equiv1(2)}} a_{4,32,3}(n) q^n.
\end{eqnarray}
Let $\chi_{-4}$ be the primitive odd character modulo $4$. Using the definition \eqref{eisenstein}, the Eisenstein series 
$E_{4,\chi_{-4},\chi_{-4}}(z)$ belongs to $M_4(\Gamma_0(16))$ and we have 
%Then the following new Eisenstein series belongs to ${\mathcal E}_4(\Gamma_0(16))$:
\begin{equation}\label{eis16} 
E_{4,\chi_{-4},\chi_{-4}}(z) = \sum_{n\ge 1} \sigma_{3,\chi_{-4},\chi_{-4}}(n) q^n = \sum_{n \ge 1} \left(\frac{-4}{n}\right) \sigma_3(n) q^n.
\end{equation}

Using the above functions, we give below a basis for the space $M_4(\Gamma_0(32))$.

\begin{prop}\label{trivial}
A basis for the space $M_4(\Gamma_0(32))$ is given by 
%% Eisenstein series space ${\mathcal E}_4(\Gamma_0(32))$ is given by 
\begin{equation}
\begin{split}
\left\{E_4(tz), t\vert 32;E_{4,\chi_{-4},\chi_{-4}}(z),E_{4,\chi_{-4},\chi_{-4}}(2z), \right. \hskip 4.5cm &\\
%\right\}  
%\end{equation}
%and a basis for the space of cusp forms $S_4(\Gamma_0(32))$ is given by 
%\begin{equation}
%\begin{split}
 \qquad \quad \left.f_{4,8}(t_1z), t_1\vert 4;  f_{4,16}(t_2z), t_2\vert 2;  f_{4,32,1}(z), f_{4,32,2}(z), f_{4,32,3}(z) \right\}.&\\
%&\\ 
\end{split}
\end{equation}
%Together they form a basis for $M_4(\Gamma_0(32))$. 
\end{prop}

\smallskip

For the sake of simplicity in the formulae, we list these basis elements as $\{F_\alpha(z)\vert 1\le\alpha\le16\}$, where  $F_1(z) = E_4(z)$, $ F_2(z) = E_4(2z) $, $ F_3(z) = E_4(4z)$,
$ F_4(z) = E_4(8z)$, $F_5(z) = E_4(16z)$, $ F_6(z) = E_4(32z)$, $ F_7(z) = E_{4,\chi_{-4},\chi_{-4}}(z)$, $ F_{8}(z) =E_{4,\chi_{-4},\chi_{-4}}(2z)$, $ F_{9}(z) = f_{4,8}(z)$, $F_{10}(z) = f_{4,8}(2z)$, $F_{11}(z) =  f_{4,8}(4z)$, $F_{12}(z) = f_{4,16}(z)$,  $F_{13}(z) = f_{4,16}(2z)$, $F_{14}(z) =  f_{4,32,1}(z)$, $F_{15}(z) =  f_{4,32,2}(z)$, $F_{16}(z) =  f_{4,32,3}(z)$. 

\noindent We also express the Fourier coefficients of the function $F_\alpha(z)= \sum_{n\ge 1} C_\alpha(n) q^n$, $1\le \alpha\le 16$. 

We are now ready to prove the theorem. Noting that all the 40 cases (corresponding to Type I in Table 2) have the property that the sum of the powers of the theta functions corresponding to the coefficients $2$ and $8$ are even. So, we can express these theta functions 
as a linear combination of the basis given in \propref{trivial} as follows. 
\begin{equation}
\Theta^i(z) \Theta^j(2z)\Theta^k(4z)\Theta^l(8z) = \sum_{\alpha=1}^{16} c_\alpha F_\alpha(z),
\end{equation}
where $a_\alpha$'s some constants. Comparing the $n$-th Fourier coefficients both the sides, we get 
\begin{equation*}
N(1^i,2^j,4^k,8^l;n) = \sum_{\alpha=1}^{16} c_\alpha C_\alpha(n).
\end{equation*}
Explicit values for the constants \!$c_\alpha$, \!$1\le\alpha\le 16$ corresponding to these 40 cases are given in Table \!3. 

\subsection{A basis for $M_4(\Gamma_0(32), \chi_8)$ and proof of \thmref{thm3}(ii)} 

The space $M_4(\Gamma_0(32), \chi_8)$ is $16$ dimensional and the cusp forms space has dimension $8$. 
For the space of Eisenstein series we use the basis elements given by \eqref{eisenstein}.  There are two Eisenstein series corresponding to 
$(\chi,\psi) = ({\bf 1}, \chi_8)$ and $(\chi,\psi) = (\chi_8, {\bf 1})$, where  $\chi_8 = \left(\frac{2}{\cdot}\right)$, the even primitive character modulo $8$. 
For the space of cusp forms, we use the following  two newforms of level $8$.
%the space dim($S_4^{new}(\Gamma_0(8))$) = 2 and is spanned by the forms $ f_{4,8,\chi_8;1}(z),f_{4,8,\chi_8;2}(z)$, defined as.
%of ${\mathcal E}_4(\Gamma_0(8),\chi_8)$ given in \eqref{eis:8chi8}. 
\begin{equation}
\begin{split}\label{new:chi8}
f_{4,8,\chi_8;1}(z) &=  1^{-2} 2^{11} 4^{-3}8^2:=  \sum_{n\ge 1} a_{4,8,\chi_8;1}(n) q^n ,~~\\
%\frac{\eta^{2}(z) \eta(2z)\eta(4z) \eta^{10}(6z)\eta^{2}(8z)}{\eta^4(3z) \eta^4(12z)} 
f_{4,8,\chi_8;2}(z) & =  1^2 2^{-3} 4^{11} 8^{-2}:= \sum_{n\ge 1} a_{4,8,\chi_8;2}(n) q^n,~~\\
%\frac{\eta(z) \eta^3(2z)\eta(4z) \eta^{4}(6z)\eta(24z)}{\eta(3z) \eta(8z)}
\end{split}
\end{equation}
We also need the $2$ newforms of level $32$, which we define below. 
%Also the space $S_4^{new}(\Gamma_0(32),\chi_8) $ is 2-dimensional and it is spanned by the forms $ f_{4,32,\chi_8;1}(z)$ and $f_{4,32,\chi_8;2}(z)$, defined as follows:
Let 
\begin{equation}
\begin{split}
g_{4,32,\chi_8;1}(z) & =  1^{2} 2^1 4^5 := \sum_{n\ge 1} a_{4,32,\chi_8;1}(n) q^n , ~~\\
g_{4,32,\chi_8;2}(z) & = 1^{-2} 2^3 4^3 8^4 := \sum_{n\ge 1} a_{4,32,\chi_8;2}(n) q^n.
\end{split}
\end{equation}
Then the two newforms of level $32$ are defined by 
\begin{equation}\label{32:chi8}
\begin{split}
f_{4,32,\chi_8;1}(z) & := \sum_{n\ge 1} \chi_4(n)  a_{4,32,\chi_8;1}(n) q^n, ~~\\
f_{4,32,\chi_8;2}(z) & := \sum_{n\ge 1} \chi_4(n) a_{4,32,\chi_8;2}(n) q^n, ~~\\
\end{split}
\end{equation}
where $\chi_4$ is the trivial character modulo $4$.

A basis for the space $M_4(\Gamma_0(32),\chi_{8})$ is given in the following proposition. 

\begin{prop}\label{chi8} 
A basis for the space $M_4(\Gamma_0(32),\chi_8)$ is given by 
\begin{equation}
\begin{split}
&\left\{E_{4, {\bf 1},\chi_8}(tz), E_{4,\chi_8, {\bf 1}}(tz), t\vert 4; E_{4,\chi_{-4},\chi_{-8}}(z),E_{4,\chi_{-8},\chi_{-4}}(z)\right.\\
&\left.f_{4,8,\chi_8;1}(t_1z), f_{4,8,\chi_8;2}(t_1z), t_1\vert 4;  f_{4,32,\chi_8;1}(z),f_{4,32,\chi_8;2}(z) \right\}.\\
\end{split}
\end{equation}
In the above, $E_{4,{\bf 1},\chi_8}(z)$ and $E_{4,\chi_8,{\bf 1}}(z)$ are defined as in \eqref{eisenstein}, $f_{4,8,\chi_8;i}(z)$, $i=1,2$ are defined in \eqref{new:chi8} and  $f_{4,32,\chi_8;j}(z)$, $1\le j\le 2$ are defined by \eqref{32:chi8}
\end{prop}

%\smallskip

For the sake of simplifying of the notation, we shall list the basis in \propref{chi8} as \\
$\left\{G_\alpha(z)\vert 1\le \alpha\le 16\right\}$, where 
$G_1(z) = E_{4, {\bf 1},\chi_8}(z)$, $G_2(z) = E_{4, {\bf 1},\chi_8}(2z)$,
$G_3(z) = E_{4, {\bf 1},\chi_8}(4z)$, $G_4(z) = E_{4,\chi_8, {\bf 1}}(z)$,  $G_5(z) = E_{4,\chi_8, {\bf 1}}(2z)$,  
$G_6(z) = E_{4,\chi_8, {\bf 1}}(4z)$, $G_7(z) = E_{4,\chi_{-4},\chi_{-8}}(z)$,  $G_8(z) =E_{4,\chi_{-8},\chi_{-4}}(z)$, 
$G_9(z) = f_{4,8,\chi_8;1}(z)$, $G_{10}(z) = f_{4,8,\chi_8;1}(2z)$, $G_{11}(z) = f_{4,8,\chi_8;1}(4z)$, $G_{12}(z) = f_{4,8,\chi_8;2}(z)$, $G_{13}(z) = f_{4,8,\chi_8;2}(2z)$, $G_{14}(z) = f_{4,8,\chi_8;2}(4z)$, $G_{15}(z) = f_{4,32,\chi_8;1}(z)$, $ G_{16}(z) = f_{4,32,\chi_8;2}(z)$.

\noindent 
As before, we also write the Fourier expansions of these basis elements as $G_\alpha(z) = \sum_{n\ge 1} D_\alpha(n) q^n$, $1\le \alpha\le 16$.

In this case, all the 44 quadruples (corresponding to Type II in Table 2) have the property that the sum of the powers of the theta functions corresponding to the coefficients $2$ and $8$ are odd. Therefore, the resulting products of theta functions are modular forms of weight $4$ on $\Gamma_0(32)$ with character $\chi_8$ (as observed earlier). So, we can express these products of theta functions as a linear combination of the basis given in \propref{chi8} as follows. 
\begin{equation}
\Theta^i(z) \Theta^j(2z)\Theta^k(3z)\Theta^l(4z) = \sum_{\alpha=1}^{16} d_\alpha G_\alpha(z).
\end{equation}
Comparing the $n$-th Fourier coefficients both the sides, we get 
\begin{equation*}
N(1^i,2^j,4^k,8^l;n) = \sum_{\alpha=1}^{16} d_\alpha D_\alpha(n).
\end{equation*}
Explicit values for the constants $d_\alpha$, $1\le \alpha\le 14$ corresponding to these 44 cases are given in Table 4.

%\bigskip

%\newpage

\section{List of tables}
In this section, we list Tables 3 and 4, which give the coefficients for the formulas for the number of representations corresponding to \thmref{thm3} 
(i)  and (ii).

\vfill

\begin{center}
\textbf{Table 3.} (Theorem 2.3 (i))
\begin{tabular}{|c|c|c|c|c|c|c|c|c|c|c|c|c|c|c|c|c|}
\hline
&&&&&&&&&&&&&&&&\\
$ijkl$&$c_1$ &$c_2$& $c_3$& $c_4$&$c_5$ &$c_6$& $c_7$& $c_8$&$c_9$ &$c_{10}$& $c_{11}$& $c_{12}$&$c_{13}$ &$c_{14}$& $c_{15}$& $c_{16}$ \\
\hline 
&&&&&&&&&&&&&&&&\\
$1016$&$\frac{1}{15360}$& $\frac{-3}{5120}$& $\frac{17}{1920}$& $\frac{-1}{120}$& $\frac{-1}{15}$& $\frac{16}{15}$& $\frac{1}{64}$&0&$\frac{31}{64}$&0&2&$\frac{31}{64}$&0&$\frac
{13}{8}$& $\frac{3}{4}$&$\frac{-5}{8}$\\
&&&&&&&&&&&&&&&&\\
$1034$&$\frac{1}{7680}$& $\frac{-3}{2560}$& $\frac{17}{960}$& $\frac{-1}{60}$& $\frac{-1}{15}$& $\frac{16}{15}$& $\frac{1}{32}$&0&$\frac{15}{32}$&0&4& $\frac{15}{32}$&0&$\frac{3}{2}$&1&$\frac{-1}{2}$\\
&&&&&&&&&&&&&&&&\\
$1052$&$\frac{1}{3840}$& $\frac{-3}{1280}$& $\frac{17}{480}$& $\frac{-1}{30}$& $\frac{-1}{15}$& $\frac{16}{15}$& $\frac{1}{16}$&0&$\frac{7}{16}$&0&4&$\frac{7}{16}$& 0 &$\frac{3}{2}$& 1 &$\frac{-1}{2}$\\
&&&&&&&&&&&&&&&&\\
$1115$&$\frac{1}{7680}$& $\frac{-1}{7680}$&0&0 &$\frac{-1}{15}$& $\frac{16}{15}$&0&0 &$\frac{11}{32}$& $\frac{3}{4}$&2& $\frac{5}{8}$& 1& $\frac{11}{8}$&
$\frac{1}{4}$& $\frac{-3}{8}$\\
&&&&&&&&&&&&&&&&\\

$1133$&$\frac{1}{3840}$& $\frac{-1}{3840}$&0&0 &$\frac{-1}{15}$& $\frac{16}{15}$&0&0 &$\frac{7}{16}$& $\frac{1}{2}$&4&  $\frac{1}{2}$& 1&  $\frac{5}{4}$&
$\frac{1}{2}$&  $\frac{-1}{4}$\\

&&&&&&&&&&&&&&&&\\
$1151$&$\frac{1}{1920}$& $\frac{-1}{1920}$&0&0 &$\frac{-1}{15}$& $\frac{16}{15}$&0&0 &$\frac{3}{8}$&1&4&  $\frac{1}{2}$& 0&  $\frac{3}{2}$&1&  $\frac{-1}{2}$\\

&&&&&&&&&&&&&&&&\\
$1214$&$\frac{1}{3840}$& $\frac{-1}{3840}$&0&0 &$\frac{-1}{15}$& $\frac{16}{15}$&0&0 &$\frac{3}{16}$& $\frac{3}{2}$&4&  $\frac{3}{4}$& 2&1&0&0\\
&&&&&&&&&&&&&&&&\\

$1232$&$\frac{1}{1920}$& $\frac{-1}{1920}$&0&0 &$\frac{-1}{15}$& $\frac{16}{15}$&0&0 &$\frac{3}{8}$&1&4&  $\frac{1}{2}$& 2&1&0&0\\

&&&&&&&&&&&&&&&&\\
$1313$&$\frac{1}{1920}$& $\frac{-1}{1920}$&0&0 &$\frac{-1}{15}$& $\frac{16}{15}$&0&0 &$\frac{1}{8}$&2&8&  $\frac{3}{4}$& 3&  $\frac{1}{2}$&0&  $\frac{1}{2}$\\

&&&&&&&&&&&&&&&&\\
$1331$&$\frac{1}{960}$& $\frac{-1}{960}$&0&0 &$\frac{-1}{15}$& $\frac{16}{15}$&0&0 &$\frac{1}{4}$&2&4&  $\frac{1}{2}$& 2&1&0&0\\

&&&&&&&&&&&&&&&&\\
$1412$&$\frac{1}{960}$& $\frac{-1}{960}$&0&0 &$\frac{-1}{15}$& $\frac{16}{15}$&0&0 &$\frac{1}{4}$&2&12&  $\frac{1}{2}$& 4&0&0&1\\
&&&&&&&&&&&&&&&&\\
$1511$&$\frac{1}{480}$& $\frac{-1}{480}$&0&0 &$\frac{-1}{15}$& $\frac{16}{15}$&0&0 &$\frac{1}{2}$&2&12&0&4&0&0&1\\
&&&&&&&&&&&&&&&&\\

$2006$&$\frac{1}{7680}$& $\frac{-1}{7680}$&0&0 &$\frac{-1}{15}$& $\frac{16}{15}$& $\frac{1}{32}$& $\frac{1}{4}$& $\frac{31}{32}$& $\frac{7}{4}$&2&  $\frac{31}{32}$&  $\frac{7}{4}$&  $\frac{13}{4}$& $\frac{3}{2}$&  $\frac{-5}{4}$\\
&&&&&&&&&&&&&&&&\\
$2024$&$\frac{1}{3840}$& $\frac{-1}{3840}$&0&0 &$\frac{-1}{15}$& $\frac{16}{15}$& $\frac{1}{16}$&0&$\frac{15}{16}$& $\frac{3}{2}$&4&  $\frac{15}{16}$& 2&3&2&-1\\
&&&&&&&&&&&&&&&&\\
$2042$&$\frac{1}{1920}$& $\frac{-1}{1920}$&0&0 &$\frac{-1}{15}$& $\frac{16}{15}$& $\frac{1}{8}$&0&$\frac{7}{8}$&1&4&  $\frac{7}{8}$& 2&3&2&-1\\
&&&&&&&&&&&&&&&&\\
$2105$&$\frac{1}{3840}$& $\frac{-1}{3840}$& $\frac{1}{120}$& $\frac{-17}{120}$& $\frac{1}{15}$& $\frac{16}{15}$&0&  $\frac{1}{4}$& $\frac{11}{16}$&  $\frac{5}{2}$&6&  $\frac{5}{4}$&  $\frac{11}{4}$&  $\frac{11}{4}$&  $\frac{1}{2}$&  $\frac{-3}{4}$\\
&&&&&&&&&&&&&&&&\\
$2123$&$\frac{1}{1920}$& $\frac{-1}{1920}$&0&0 &$\frac{-1}{15}$& $\frac{16}{15}$&0&0 &$\frac{7}{8}$&2&8&1&3&$\frac{5}{2}$&1&  $\frac{-1}{2}$\\
&&&&&&&&&&&&&&&&\\
$2141$&$\frac{1}{960}$& $\frac{-1}{960}$&0&0 &$\frac{-1}{15}$& $\frac{16}{15}$&0&0 &$\frac{3}{4}$&2&4&1&2&3&2&-1\\
&&&&&&&&&&&&&&&&\\
$2204$&$\frac{1}{1920}$& $\frac{-1}{1920}$& $\frac{1}{60}$& $\frac{-17}{60}$& $\frac{1}{5}$& $\frac{16}{15}$&0&0 &$\frac{3}{8}$&3&12&  $\frac{3}{2}$& 4&2&0&0\\
&&&&&&&&&&&&&&&&\\
$2222$&$\frac{1}{960}$& $\frac{-1}{960}$&0&0 &$\frac{-1}{15}$& $\frac{16}{15}$&0&0 &$\frac{3}{4}$&2&12&1&4&2&0&0\\
&&&&&&&&&&&&&&&&\\
\hline
\end{tabular}
%\end{landscape}
\end{center}

\newpage
\begin{center}
\textbf{Table 3.} (Theorem 2.3 (i))(contd.)
\begin{tabular}{|c|c|c|c|c|c|c|c|c|c|c|c|c|c|c|c|c|}
\hline
&&&&&&&&&&&&&&&&\\
$ijkl$&$c_1$ &$c_2$& $c_3$& $c_4$&$c_5$ &$c_6$& $c_7$& $c_8$&$c_9$ &$c_{10}$& $c_{11}$& $c_{12}$&$c_{13}$ &$c_{14}$& $c_{15}$& $c_{16}$ \\
\hline 
&&&&&&&&&&&&&&&&\\
$2303$&$\frac{1}{960}$& $\frac{-1}{960}$& $\frac{1}{60}$& $\frac{-17}{60}$& $\frac{1}{5}$& $\frac{16}{15}$&0&  $\frac{-1}{2}$& $\frac{1}{4}$&3&20&  $\frac{3}{2}$&  $\frac{11}{2}$&1&0&1\\
&&&&&&&&&&&&&&&&\\
$2321$&$\frac{1}{480}$& $\frac{-1}{480}$&0&0 &$\frac{-1}{15}$& $\frac{16}{15}$&0&0 &$\frac{1}{2}$&2&12&1&4&2&0&0\\
&&&&&&&&&&&&&&&&\\
$2402$&$\frac{1}{480}$& $\frac{-1}{480}$&0&0 &$\frac{-1}{15}$& $\frac{16}{15}$&0&-1&  $\frac{1}{2}$&2&28&1&7&0&0&2\\
&&&&&&&&&&&&&&&&\\
$2501$&$\frac{1}{240}$& $\frac{-1}{240}$& $\frac{-1}{30}$& $\frac{17}{30}$& $\frac{-3}{5}$& $\frac{16}{15}$&0&-1&1&0&28&0&7&0&0&2\\
&&&&&&&&&&&&&&&&\\
$3014$&$\frac{1}{1920}$& $\frac{1}{640}$& $\frac{-17}{480}$& $\frac{1}{30}$& $\frac{-1}{15}$& $\frac{16}{15}$& $\frac{1}{16}$&0&$\frac{9}{8}$& $\frac{9}{2}$&4&  $\frac{27}{16}$&6&4&2&-1\\
&&&&&&&&&&&&&&&&\\
$3032$&$\frac{1}{960}$& $\frac{1}{320}$& $\frac{-17}{240}$& $\frac{1}{15}$& $\frac{-1}{15}$& $\frac{16}{15}$& $\frac{1}{8}$&0&$\frac{5}{4}$&3&4&  $\frac{11}{8}$&6&4&2&-1\\
&&&&&&&&&&&&&&&&\\
$3113$&$\frac{1}{960}$& $\frac{-1}{960}$&0&0 &$\frac{-1}{15}$& $\frac{16}{15}$&0&0&1&5&16&  $\frac{7}{4}$&7&3&1&0\\
&&&&&&&&&&&&&&&&\\
$3131$&$\frac{1}{480}$& $\frac{-1}{480}$&0&0 &$\frac{-1}{15}$& $\frac{16}{15}$&0&0&1&4&4&  $\frac{3}{2}$&6&4&2&-1\\
&&&&&&&&&&&&&&&&\\
$3212$&$\frac{1}{480}$& $\frac{-1}{480}$&0&0 &$\frac{-1}{15}$& $\frac{16}{15}$&0&0&1&4&28&  $\frac{3}{2}$&8&2&0&1\\
&&&&&&&&&&&&&&&&\\
$3311$&$\frac{1}{240}$& $\frac{-1}{240}$&0&0 &$\frac{-1}{15}$& $\frac{16}{15}$&0&0&1&2&28&1&8&2&0&1\\
&&&&&&&&&&&&&&&&\\
$4004$&$\frac{1}{960}$& $\frac{1}{320}$& $\frac{-13}{240}$& $\frac{-13}{60}$& $\frac{1}{5}$& $\frac{16}{15}$&0&0 &$\frac{3}{4}$&9&12&3&12&4&0&0\\
&&&&&&&&&&&&&&&&\\
$4022$&$\frac{1}{480}$& $\frac{1}{160}$& $\frac{-17}{120}$& $\frac{2}{15}$& $\frac{-1}{15}$& $\frac{16}{15}$&0&0 &$\frac{3}{2}$&6&12&2&12&4&0&0\\
&&&&&&&&&&&&&&&&\\
$4103$&$\frac{1}{480}$& $\frac{-1}{480}$& $\frac{1}{60}$& $\frac{-17}{60}$& $\frac{1}{5}$& $\frac{16}{15}$& 0&$\frac{-1}{2}$& $\frac{1}{2}$&9&36&3&$\frac{27}{2}$&2&0&2\\
&&&&&&&&&&&&&&&&\\
$4121$&$\frac{1}{240}$& $\frac{-1}{240}$&0&0 &$\frac{-1}{15}$& $\frac{16}{15}$&0&0&1&6&12&2&12&4&0&0\\
&&&&&&&&&&&&&&&&\\
$4202$&$\frac{1}{240}$& $\frac{-1}{240}$&0&0 &$\frac{-1}{15}$& $\frac{16}{15}$&0&-1&1&6&60&2&15&0&0&4\\
&&&&&&&&&&&&&&&&\\
$4301$&$\frac{1}{120}$& $\frac{-1}{120}$& $\frac{-1}{30}$& $\frac{17}{30}$& $\frac{-3}{5}$& $\frac{16}{15}$&0&-1&2&0&60&0&15&0&0&4\\
&&&&&&&&&&&&&&&&\\
$5012$&$\frac{1}{240}$& $\frac{1}{240}$& $\frac{-17}{120}$& $\frac{2}{15}$& $\frac{-1}{15}$& $\frac{16}{15}$& $\frac{-1}{4}$&0&$\frac{3}{2}$&10&44&$\frac{11}{4}$&20&2&-4&3\\
&&&&&&&&&&&&&&&&\\
$5111$&$\frac{1}{120}$& $\frac{-1}{120}$&0&0 &$\frac{-1}{15}$& $\frac{16}{15}$&0&0&1&6&44&2&20&2&-4&3\\
&&&&&&&&&&&&&&&&\\
$6002$&$\frac{1}{120}$& $\frac{-1}{120}$&0&0 &$\frac{-1}{15}$& $\frac{16}{15}$& $\frac{-1}{2}$&-1&1&14&124&$\frac{7}{2}$&31&-4&-8&10\\
&&&&&&&&&&&&&&&&\\
$6101$&$\frac{1}{60}$& $\frac{-1}{60}$& $\frac{-1}{30}$& $\frac{17}{30}$& $\frac{-3}{5}$& $\frac{16}{15}$&0&-1&2& 0&124&0&31&-4&-8&10\\
&&&&&&&&&&&&&&&&\\
\hline
\end{tabular}
%\end{landscape}
\end{center}

\newpage
\begin{center}
\textbf{Table 4.} (Theorem 2.3 (ii))
\begin{tabular}{|c|c|c|c|c|c|c|c|c|c|c|c|c|c|c|c|c|}
\hline
&&&&&&&&&&&&&&&&\\
$ijkl$&$d_1$ &$d_2$& $d_3$& $d_4$&$d_5$ &$d_6$& $d_7$& $d_8$&$d_9$ &$d_{10}$& $d_{11}$& $d_{12}$&$d_{13}$ &$d_{14}$& $d_{15}$& $d_{16}$  \\
\hline 
&&&&&&&&&&&&&&&&\\
$1007$&$\frac{1}{88}$& $\frac{-1}{88}$& $\frac{2}{11}$& $\frac{1}{88}$& $\frac{-1}{11}$& $\frac{16}{11}$& $\frac{1}{88}$& $\frac{1}{88}$& $\frac{43}{176}$& $\frac{43}{22}$& $\frac{8}{11}$& $\frac{129}{176}$& $\frac{-43}{44}$& $\frac{-4}{11}$& $\frac{43}{44}$& $\frac{43}{44}$\\
&&&&&&&&&&&&&&&&\\
$1025$&0&0& $\frac{2}{11}$& $\frac{1}{44}$& $\frac{-2}{11}$& $\frac{32}{11}$&0& $\frac{1}{44}$& $\frac{7}{22}$& $\frac{43}{22}$& $\frac{12}{11}$& $\frac{29}{44}$& $\frac{-14}{11}$& $\frac{20}{11}$& $\frac{43}{44}$& $\frac{14}{11}$\\
&&&&&&&&&&&&&&&&\\
$1043$&0&0& $\frac{2}{11}$& $\frac{1}{22}$& $\frac{-4}{11}$& $\frac{64}{11}$&0& $\frac{1}{22}$& $\frac{17}{44}$& $\frac{21}{11}$& $\frac{20}{11}$& $\frac{25}{44}$& $\frac{-17}{11}$& $\frac{24}{11}$& $\frac{21}{22}$& $\frac{17}{11}$\\
&&&&&&&&&&&&&&&&\\
$1061$&0&0& $\frac{2}{11}$& $\frac{1}{11}$& $\frac{-8}{11}$& $\frac{128}{11}$&0& $\frac{1}{11}$& $\frac{3}{11}$& $\frac{20}{11}$& $\frac{-8}{11}$& $\frac{7}{11}$& $\frac{-12}{11}$& $\frac{32}{11}$& $\frac{10}{11}$& $\frac{12}{11}$\\
&&&&&&&&&&&&&&&&\\
$1106$&0&0& $\frac{2}{11}$& $\frac{1}{44}$&0&0& $\frac{1}{44}$&0& $\frac{3}{44}$& $\frac{5}{2}$& $\frac{48}{11}$& $\frac{10}{11}$&1& $\frac{-28}{11}$& $\frac{43}{44}$& $\frac{37}{22}$\\
&&&&&&&&&&&&&&&&\\
$1124$&0&0& $\frac{2}{11}$& $\frac{1}{22}$&0&0&0&0& $\frac{3}{22}$&2& $\frac{48}{11}$& $\frac{9}{11}$&1& $\frac{16}{11}$&1&2\\
&&&&&&&&&&&&&&&&\\
$1142$&0&0& $\frac{2}{11}$& $\frac{1}{11}$&0&0&0&0& $\frac{3}{11}$&2& $\frac{48}{11}$& $\frac{7}{11}$&0& $\frac{16}{11}$&1&2\\
&&&&&&&&&&&&&&&&\\
$1205$&$\frac{-1}{44}$& $\frac{1}{44}$& $\frac{2}{11}$& $\frac{1}{22}$&0&0& $\frac{1}{44}$&0& $\frac{-3}{88}$& $\frac{67}{22}$& $\frac{92}{11}$& $\frac{89}{88}$& $\frac{59}{22}$& $\frac{-28}{11}$& $\frac{43}{44}$& $\frac{59}{22}$\\
&&&&&&&&&&&&&&&&\\
$1223$&0&0& $\frac{2}{11}$& $\frac{1}{11}$&0&0&0&0& $\frac{1}{44}$&2& $\frac{92}{11}$& $\frac{39}{44}$&3& $\frac{16}{11}$&1&3\\
&&&&&&&&&&&&&&&&\\
$1241$&0&0& $\frac{2}{11}$& $\frac{2}{11}$&0&0&0&0& $\frac{1}{22}$&2& $\frac{48}{11}$& $\frac{17}{22}$&2& $\frac{16}{11}$&1&2\\
&&&&&&&&&&&&&&&&\\
$1304$&$\frac{-1}{22}$& $\frac{1}{22}$& $\frac{2}{11}$& $\frac{1}{11}$&0&0&0&0& $\frac{-3}{44}$& $\frac{34}{11}$& $\frac{136}{11}$& $\frac{45}{44}$& $\frac{48}{11}$& $\frac{16}{11}$&1&4\\
&&&&&&&&&&&&&&&&\\
$1322$&0&0& $\frac{2}{11}$& $\frac{2}{11}$&0&0&0&0& $\frac{1}{22}$&2& $\frac{136}{11}$& $\frac{17}{22}$&4& $\frac{16}{11}$&1&4\\
&&&&&&&&&&&&&&&&\\
$1403$&$\frac{-1}{22}$& $\frac{1}{22}$& $\frac{2}{11}$& $\frac{2}{11}$&0&0& $\frac{-1}{22}$&0& $\frac{-1}{22}$& $\frac{23}{11}$& $\frac{180}{11}$& $\frac{10}{11}$& $\frac{70}{11}$& $\frac{104}{11}$& $\frac{23}{22}$& $\frac{62}{11}$\\
&&&&&&&&&&&&&&&&\\
$1421$&0&0& $\frac{2}{11}$& $\frac{4}{11}$&0&0&0&0& $\frac{1}{11}$&2& $\frac{136}{11}$& $\frac{6}{11}$&4& $\frac{16}{11}$&1&4\\
&&&&&&&&&&&&&&&&\\
$1502$&0&0& $\frac{2}{11}$& $\frac{4}{11}$&0&0& $\frac{-1}{11}$&0& $\frac{1}{11}$&0& $\frac{224}{11}$& $\frac{6}{11}$&8& $\frac{192}{11}$& $\frac{12}{11}$& $\frac{80}{11}$\\
&&&&&&&&&&&&&&&&\\
$1601$&$\frac{1}{11}$& $\frac{-1}{11}$& $\frac{2}{11}$& $\frac{8}{11}$&0&0& $\frac{-1}{11}$&0& $\frac{4}{11}$& $\frac{-24}{11}$& $\frac{224}{11}$& $\frac{-2}{11}$& $\frac{80}{11}$& $\frac{192}{11}$& $\frac{12}{11}$& $\frac{80}{11}$\\
&&&&&&&&&&&&&&&&\\
$2015$&0&0& $\frac{2}{11}$& $\frac{1}{22}$&0&0&0& $\frac{1}{22}$& $\frac{7}{11}$&5& $\frac{92}{11}$& $\frac{29}{22}$&0& $\frac{-28}{11}$& $\frac{43}{22}$& $\frac{28}{11}$\\
&&&&&&&&&&&&&&&&\\
$2033$&0&0& $\frac{2}{11}$& $\frac{1}{11}$&0&0&0& $\frac{1}{11}$& $\frac{17}{22}$&4& $\frac{92}{11}$& $\frac{25}{22}$&0& $\frac{16}{11}$& $\frac{21}{11}$& $\frac{34}{11}$\\
&&&&&&&&&&&&&&&&\\
$2051$&0&0& $\frac{2}{11}$& $\frac{2}{11}$&0&0&0& $\frac{2}{11}$& $\frac{6}{11}$&4& $\frac{48}{11}$& $\frac{14}{11}$&0& $\frac{16}{11}$& $\frac{20}{11}$& $\frac{24}{11}$\\
&&&&&&&&&&&&&&&&\\

&&&&&&&&&&&&&&&&\\
\hline
\end{tabular}
%\end{landscape}
\end{center}

\newpage
\begin{center}
\textbf{Table 4.} (Theorem 2.3 (ii))(contd.)
\begin{tabular}{|c|c|c|c|c|c|c|c|c|c|c|c|c|c|c|c|c|}
\hline
&&&&&&&&&&&&&&&&\\
$ijkl$&$d_1$ &$d_2$& $d_3$& $d_4$&$d_5$ &$d_6$& $d_7$& $d_8$&$d_9$ &$d_{10}$& $d_{11}$& $d_{12}$&$d_{13}$ &$d_{14}$& $d_{15}$& $d_{16}$ \\
\hline 
&&&&&&&&&&&&&&&&\\
$2114$&0&0& $\frac{2}{11}$& $\frac{1}{11}$&0&0&0&0& $\frac{3}{11}$&5& $\frac{136}{11}$& $\frac{18}{11}$&3& $\frac{16}{11}$&2&4\\
&&&&&&&&&&&&&&&&\\
$2132$&0&0& $\frac{2}{11}$& $\frac{2}{11}$&0&0&0&0& $\frac{6}{11}$&4& $\frac{136}{11}$& $\frac{14}{11}$&2& $\frac{16}{11}$&2&4\\
&&&&&&&&&&&&&&&&\\
$2213$&0&0& $\frac{2}{11}$& $\frac{2}{11}$&0&0&0&0& $\frac{1}{22}$&4& $\frac{180}{11}$& $\frac{39}{22}$&6& $\frac{104}{11}$&2&6\\
&&&&&&&&&&&&&&&&\\
$2231$&0&0& $\frac{2}{11}$& $\frac{4}{11}$&0&0&0&0& $\frac{1}{11}$&4& $\frac{136}{11}$& $\frac{17}{11}$&4& $\frac{16}{11}$&2&4\\
&&&&&&&&&&&&&&&&\\
$2312$&0&0& $\frac{2}{11}$& $\frac{4}{11}$&0&0&0&0& $\frac{1}{11}$&2& $\frac{224}{11}$& $\frac{17}{11}$&8& $\frac{192}{11}$&2&8\\
&&&&&&&&&&&&&&&&\\
$2411$&0&0& $\frac{2}{11}$& $\frac{8}{11}$&0&0&0&0& $\frac{2}{11}$&0& $\frac{224}{11}$& $\frac{12}{11}$&8& $\frac{192}{11}$&2&8\\
&&&&&&&&&&&&&&&&\\
$3005$&$\frac{-1}{44}$& $\frac{1}{44}$& $\frac{2}{11}$& $\frac{1}{11}$& $\frac{4}{11}$& $\frac{-64}{11}$& $\frac{1}{44}$& $\frac{1}{22}$& $\frac{53}{88}$& $\frac{201}{22}$& $\frac{252}{11}$& $\frac{205}{88}$& $\frac{115}{22}$& $\frac{-124}{11}$& $\frac{129}{44}$& $\frac{115}{22}$\\
&&&&&&&&&&&&&&&&\\
$3023$&0&0& $\frac{2}{11}$& $\frac{2}{11}$& $\frac{8}{11}$& $\frac{-128}{11}$&0& $\frac{1}{11}$& $\frac{35}{44}$& $\frac{68}{11}$& $\frac{236}{11}$& $\frac{89}{44}$& $\frac{67}{11}$&0& $\frac{32}{11}$& $\frac{67}{11}$\\
&&&&&&&&&&&&&&&&\\
$3041$&0&0& $\frac{2}{11}$& $\frac{4}{11}$& $\frac{16}{11}$& $\frac{-256}{11}$&0& $\frac{2}{11}$& $\frac{13}{22}$& $\frac{70}{11}$& $\frac{160}{11}$& $\frac{45}{22}$& $\frac{46}{11}$& $\frac{-16}{11}$& $\frac{31}{11}$& $\frac{46}{11}$\\
&&&&&&&&&&&&&&&&\\
$3104$&$\frac{-1}{22}$& $\frac{1}{22}$& $\frac{2}{11}$& $\frac{2}{11}$&0&0&0&0& $\frac{9}{44}$& $\frac{100}{11}$& $\frac{312}{11}$& $\frac{117}{44}$& $\frac{92}{11}$& $\frac{16}{11}$&3&8\\
&&&&&&&&&&&&&&&&\\
$3122$&0&0& $\frac{2}{11}$& $\frac{4}{11}$&0&0&0&0& $\frac{13}{22}$&6& $\frac{312}{11}$& $\frac{45}{22}$&8& $\frac{16}{11}$&3&8\\
&&&&&&&&&&&&&&&&\\
$3203$&$\frac{-1}{22}$& $\frac{1}{22}$& $\frac{2}{11}$& $\frac{4}{11}$&0&0& $\frac{-1}{22}$&0&0& $\frac{67}{11}$& $\frac{356}{11}$& $\frac{59}{22}$& $\frac{136}{11}$& $\frac{280}{11}$& $\frac{67}{22}$& $\frac{128}{11}$\\
&&&&&&&&&&&&&&&&\\
$3221$&0&0& $\frac{2}{11}$& $\frac{8}{11}$&0&0&0&0& $\frac{2}{11}$&6& $\frac{312}{11}$& $\frac{23}{11}$&8& $\frac{16}{11}$&3& 8\\
&&&&&&&&&&&&&&&&\\
$3302$&0&0& $\frac{2}{11}$& $\frac{8}{11}$&0&0& $\frac{-1}{11}$&0& $\frac{2}{11}$&0& $\frac{400}{11}$& $\frac{23}{11}$&16& $\frac{544}{11}$& $\frac{34}{11}$& $\frac{168}{11}$\\
&&&&&&&&&&&&&&&&\\
$3401$&$\frac{1}{11}$& $\frac{-1}{11}$& $\frac{2}{11}$& $\frac{16}{11}$&0&0& $\frac{-1}{11}$&0& $\frac{6}{11}$& $\frac{-68}{11}$& $\frac{400}{11}$& $\frac{10}{11}$& $\frac{168}{11}$& $\frac{544}{11}$& $\frac{34}{11}$& $\frac{168}{11}$\\
&&&&&&&&&&&&&&&&\\
$4013$&0&0& $\frac{2}{11}$& $\frac{4}{11}$& $\frac{16}{11}$& $\frac{-256}{11}$&0&0& $\frac{1}{11}$& $\frac{92}{11}$& $\frac{468}{11}$& $\frac{39}{11}$& $\frac{200}{11}$& $\frac{72}{11}$&4&12\\
&&&&&&&&&&&&&&&&\\
$4031$&0&0& $\frac{2}{11}$& $\frac{8}{11}$& $\frac{32}{11}$& $\frac{-512}{11}$&0&0& $\frac{2}{11}$& $\frac{96}{11}$& $\frac{360}{11}$& $\frac{34}{11}$& $\frac{136}{11}$& $\frac{-48}{11}$&4&8\\
&&&&&&&&&&&&&&&&\\
$4112$&0&0& $\frac{2}{11}$& $\frac{8}{11}$&0&0&0&0& $\frac{2}{11}$&6& $\frac{576}{11}$& $\frac{34}{11}$&20& $\frac{192}{11}$&4&16\\
&&&&&&&&&&&&&&&&\\
$4211$&0&0& $\frac{2}{11}$& $\frac{16}{11}$&0&0&0&0& $\frac{4}{11}$&4& $\frac{576}{11}$& $\frac{24}{11}$&16& $\frac{192}{11}$&4&16\\
&&&&&&&&&&&&&&&&\\
$5003$&$\frac{-1}{22}$& $\frac{1}{22}$& $\frac{2}{11}$& $\frac{8}{11}$& $\frac{16}{11}$& $\frac{-256}{11}$& $\frac{-1}{22}$& $\frac{-2}{11}$& $\frac{-31}{22}$& $\frac{115}{11}$& $\frac{820}{11}$& $\frac{63}{11}$& $\frac{402}{11}$& $\frac{424}{11}$& $\frac{115}{22}$& $\frac{258}{11}$\\
&&&&&&&&&&&&&&&&\\
$5021$&0&0& $\frac{2}{11}$& $\frac{16}{11}$& $\frac{32}{11}$& $\frac{-512}{11}$&0& $\frac{-4}{11}$& $\frac{-7}{11}$& $\frac{118}{11}$& $\frac{712}{11}$& $\frac{46}{11}$& $\frac{268}{11}$& $\frac{-48}{11}$& $\frac{9}{11}$& $\frac{172}{11}$\\
&&&&&&&&&&&&&&&&\\
$5102$&0&0& $\frac{2}{11}$& $\frac{16}{11}$&0&0& $\frac{-1}{11}$&0& $\frac{-7}{11}$&0& $\frac{928}{11}$& $\frac{46}{11}$&40& $\frac{896}{11}$& $\frac{56}{11}$& $\frac{344}{11}$\\
&&&&&&&&&&&&&&&&\\
\hline
\end{tabular}
%\end{landscape}
\end{center}

\newpage

\section*{Acknowledgements}
We have used the open-source mathematics software SAGE (www.sagemath.org) to perform our calculations. Part of the work was done when the last named author visited the School of Mathematical Sciences, NISER, Bhubaneswar. He thanks the school for their hospitality and support. 
%The second author is partially funded by SERB grant SR/FTP/MS-053/2012. 

%\bigskip

\end{document}